\documentclass[10pt,leqno]{amsart}
\usepackage{amscd}
\usepackage{color}
\usepackage{amssymb}
\usepackage{amsfonts}
\usepackage{latexsym}
\usepackage{verbatim}

\theoremstyle{plain}
\newtheorem{theorem}{Theorem}[section]

\newtheorem{prop}[theorem]{Proposition}
\newtheorem{cor}[theorem]{Corollary}

\renewcommand{\b}{\begin{equation}}
\newcommand{\e}{\end{equation}}
\newcommand{\g}{\mathfrak{g}}
\newcommand\C{{\mathbb C}}
\newcommand\R{{\mathbb R}}

\title{On the pluriclosed flow on Oeljeklaus-Toma manifolds}
\thanks{This work was supported by GNSAGA of INdAM}
\subjclass[2020]{53E30, 53C55}
\address{Dipartimento di Matematica G. Peano \\ Universit\`a di Torino\\
Via Carlo Alberto 10\\
10123 Torino\\ Italy}
\email{elia.fusi@unito.it}
\email{luigi.vezzoni@unito.it}
\author{Elia Fusi and Luigi Vezzoni}
\date{\today}
\begin{document}
\maketitle
\begin{abstract}
We investigate the pluriclosed flow on Oeljeklaus-Toma manifolds. We parametrize left-invariant   pluriclosed metrics on  Oeljeklaus-Toma manifolds and we classify the ones  which lift to an algebraic soliton of the pluriclosed flow on the universal covering. We further show that the pluriclosed flow  starting from a left-invariant pluriclosed metric has a long-time solution $\omega_t$ which once normalized collapses to a torus in the Gromov-Hausdorff sense. Moreover the lift of $\tfrac{1}{1+t}\omega_t$ to the universal covering of the manifold converges in the Cheeger-Gromov sense to $(\mathbb H^s\times\mathbb C^s, \tilde{\omega}_{\infty})$  where $\tilde{\omega}_{\infty}$  is an algebraic soliton.
\end{abstract}

\section{Introduction}

 Oeljeklaus-Toma manifolds are a very interesting class of complex manifolds introduced and firstly studied in \cite{OT}. 
These manifolds are defined as compact quotients of the type 
$$
M=\frac{\mathbb H^r\times\mathbb C^s}{U \ltimes \mathcal O_{\mathbb{K}}}
$$
where $\mathbb{H}\subseteq \C$ is the upper half-plane, $\mathcal O_{\mathbb{K}}$ is the ring of algebraic integers of an algebraic extension $\mathbb K$ of $\mathbb Q$ satisfying 
$[\mathbb K : \mathbb Q]=r+2s$ and $U$ is a free subgroup of rank $r$ of $\mathcal O^{*,+}_{\mathbb{K}}$ 
satisfying some compatible conditions.  The action of $U \ltimes \mathcal O_{\mathbb{K}}$ on $\mathbb H^r\times\mathbb C^s$ is defined via some embeddings of $\mathbb K$ in $\mathbb R$ and $\mathbb C$. Oeljeklaus-Toma manifolds have a rich geometric structure. For instance, they have a natural structure of $\mathbb T^{r+2s}$-torus bundle over a $\mathbb T^r$ and a structure of solvmanifold \cite{Kas}, i.e. they are always compact quotients of a solvable Lie group by a lattice.  
The Poincar\'e metric\footnote{In the whole paper we identify a Hermitian metric with its fundamental form.} $\omega_{\mathbb H^r}=\sqrt{-1}\sum_{a=1}^r\frac{dz_a\wedge d\bar z_a}{4(\Im\mathfrak m z_a)^2}$ induces a degenerate metric $\omega_\infty$ on $M$ which has a central role in the study of geometric flows on these manifolds.
The pair $(r,s)$ is called the {\em type} of the manifold.  
The case of type $(r,s)=(1,1)$ corresponds to the Inoue-Bombieri surfaces \cite{In}.  
 
\medskip 
In \cite{AT,FTWZ, TWCom,Zheng}  the Chern-Ricci flow \cite{Gill,TWJDG} on Oeljeklaus-Toma manifolds $M$ of type $(r,1)$ is studied.  Accordingly to the results in  \cite{AT,FTWZ,TWCom,Zheng}, under some assumptions on the initial Hermitian metric, the flow has a long-time solution $\omega_t$ such that $(M,\frac{\omega_t}{1+t})$ converges in the Gromov-Hausdorff sense to an $r$-dimensional torus $\mathbb T^r$ as $t\to \infty$. The result can be adapted 
to Oeljeklaus-Toma manifolds of arbitrary type by assuming the initial metric to be left-invariant with respect to the structure of solvmanifold. 
Moreover, a result of Lauret in \cite{lauret} allows us to give a characterization of left-invariant Hermitian metrics on an Oeljeklaus-Toma  manifold which lift to an algebraic soliton of the Chern-Ricci flow on the universal covering of the manifold (see Proposition \ref{main1} in the present paper).   

\medskip 
Following the same approach, we focus on the pluriclosed flow on Oeljeklaus-Toma manifolds when the initial pluriclosed Hermitian metric is left-invariant. The pluriclosed flow is a geometric flow of pluriclosed metrics, i.e. of Hermitian metrics having the fundamental form $\partial \bar\partial$-closed, 
introduced by Streets and Tian in \cite{streets-tian2}. The flow belongs to the family of  the Hermitian curvature flows \cite{streets-tian} and evolves an initial pluriclosed metric along the $(1,1)$-component of the Bismut-Ricci form. Namely, on a Hermitian manifold $(M,\omega)$ 
there always exists a unique metric connection $\nabla^B$, called the {\em Bismut connection},  preserving the complex structure and such that 
$$
\omega(T^B(\cdot,\cdot),J\cdot)\quad \mbox{is a $3$-form}\,,
$$  
where $T^B$ is the torsion of $\nabla^B$. The {\em Bismut-Ricci form} of $\omega$ is then defined as 
$$
\rho_B(X,Y):=\sqrt{-1}\sum_{i=1}^n R_B(X,Y,X_i,\bar X_i)\,,
$$
where $R_B$ is the curvature tensor of $\nabla^B$ and $\{X_i\}$ is a unitary frame of $\omega$. $\rho_B$ is always a closed real form.   
Given  a pluriclosed Hermitian metric $\omega$ on $M$, the {\em pluriclosed flow} is then defined as the geometric flow of pluriclosed metrics governed by the equation 
$$
\partial_t\omega_t=-\rho_B^{1,1}(\omega_t)\,,\quad \omega_{|t=0}=\omega\,.
$$

The pluriclosed  flow was  deeply studied in literature, see for instance \cite{AL, boling, EFV2, mario, JS,  PV, streets-tian4, Streets3, Streets1, Streets4, Streets2, Streets} and the references therein.

\medskip 
Our main result is the following

\begin{theorem}\label{main2}
Let $\omega$  be a left-invariant pluriclosed Hermitian metric on an Oeljeklaus-Toma manifold $M$. Then the pluriclosed flow starting from $\omega$ has a long-time solution $\omega_t$ such that $(M,
\frac{\omega_t}{1+t})$ converges in the Gromov-Hausdorff sense to $(\mathbb T^s,d)$. Moreover, $\omega$ lifts to an expanding algebraic soliton on the universal covering of $M$ if and only if it is diagonal and  the first $s$ diagonal components coincide. Finally, $(\mathbb H^s\times\mathbb C^s, \frac{\omega_t}{1+t})$ converges in the Cheeger-Gromov sense to $(\mathbb H^s\times\mathbb C^s, \tilde{\omega}_{\infty})$  where $\tilde{\omega}_{\infty}$  is an algebraic soliton.
\end{theorem}
Here we recall that a left-invariant Hermitian metric $\omega$ on a Lie group $G$ with 
a left-invariant complex structure is an {\em algebraic soliton} for a geometric flow    
of left-invariant Hermitian metrics if $\omega_t=c_t\varphi_t^*(\omega)$ solves the flow, where $\{c_t\}$ is a positive scaling and 
$\{\varphi_t\}$ is a family of automorphims of $G$ preserving the complex structure. 
Moreover the distance $d$ in the statement is the distance induced by  $3\omega_{\infty}$ on the torus base of $M$. Now we describe the condition {\em diagonal} appearing in the statement of Theorem \ref{main2}. The existence of a pluriclosed metric on an Oeljeklaus-Toma manifold imposes some restrictions, see \cite[Corollary 3]{A}. In particular, the manifold has type $(s,s)$ and  admits a left-invariant $(1,0)$-coframe 
$\{\omega^1,\dots,\omega^s,\gamma^1,\dots,\gamma^s\}$ satisfying 
\begin{equation*}
\begin{cases}
d\omega^k=\frac{\sqrt{-1}}{2}\omega^k\wedge\bar \omega^k\quad & k=1,\dots,s\,,\\
d\gamma^i=\sum_{k=1}^s\lambda_{ki}\,\omega^k\wedge \gamma^{i}-\sum_{k=1}^s\,\lambda_{ki}\bar{\omega}^{k}\wedge\gamma^i\quad& i=1,\dots, s\,,
\end{cases}
\end{equation*}
with 
$$
 \Im\mathfrak m\, \lambda_{ki}=-\frac14\,\delta_{ik} \,.
 $$
By $\omega$ {\em diagonal} we mean that it takes a diagonal form with respect to such a coframe.  The first part of Theorem \ref{main2} in the case of the Inoue-Bombieri surfaces is proved in \cite[Corollary 3.18]{boling}.

Theorem  \ref{main2} provides a description of the long-time behavior of the solution $\omega_t$ to the pluriclosed flow as $t\to \infty$. For the definition of the convergence in the Gromov-Hausdorff sense we refer to Section 3 in the preset paper, while here we briefly recall the definition of convergence in the Cheeger-Gromov sense:\\
a sequence of pointed riemannian manifolds $(M_k, g_k, p_k)$ {\em converges in the Cheeger-Gromov sense to a pointed riemannian manifold $(M,g,p)$} if there exists a sequence of open subsets $A_k$ of $M$  so that every compact subset of $M$ eventually lies in some $A_k$,  and a sequence of smooth maps $\phi_k\colon A_k\to M_k $ which are diffeomorphisms onto some open set of $M_k$ which satisfy $\phi_k(p_k)=p$, such that 
  $$
  \phi_k^*(g_k)\to g\quad \mbox{smoothly on every  compact subset, as } k\to \infty.
  $$
See \cite[Section 6]{LauretCG} for a deep analysis of Cheeger-Gromov convergence both in the general case and in the homogeneous one and \cite[Section 5.1]{lauret} for the case of Hermitian Lie groups. 

\bigskip 
\noindent {\bf Acknowledgment.} Authors are grateful to Daniele Angella, Ramiro Lafuente, Francesco Pediconi and Alberto Raffero for useful conversations. In particular Ramiro Lafuente suggested us how to prove the convergence in the  Cheeger-Gromov  sense in Theorem \ref{main2}.

\medskip 

\section{Definition of Oeljeklaus-Toma manifolds}
We briefly recall the construction of Oeljeklaus-Toma manifolds \cite{OT}. 
 
 \medskip 
Let $\mathbb{Q}\subseteq\mathbb{K}$ be an algebraic number field with $[\mathbb{K}:\mathbb{Q}]=r+2s$ and $r,s\ge 1$. 
Let $\sigma_1,\ldots, \sigma_r\colon\mathbb{K}\to \mathbb{R}$  be the real embeddings of $\mathbb{K}$ and   $\sigma_{r+1},\ldots,\sigma_{r+2s}\colon \mathbb{K}\to \mathbb{C}$ be the complex embeddings of $\mathbb{K}$ satisfying $\sigma_{r+s+i}=\bar \sigma_{r+i},$ for every $ i=1,\ldots, s.$ We denote by $\mathcal{O}_{\mathbb{K}}$ the ring of algebraic integers of $\mathbb{K}$ and  by $\mathcal{O}_{\mathbb{K}}^*$
the group of units of $\mathcal{O}_{\mathbb{K}}$. Let 
$$
\mathcal{O}_{\mathbb{K}}^{*,+}=\{u\in\mathcal{O}_{\mathbb{K}}^* \quad | \quad \sigma_i(u)>0\, , \quad \mbox{for every } i=1,\ldots, r\}
$$
be the group of totally positive units of $\mathcal{O}_{\mathbb{K}}$.
The groups $\mathcal{O}_{\mathbb{K}}$ and $\mathcal{O}_{\mathbb{K}}^{*,+}$ act on $\mathbb H^r\times \mathbb C^s$ as  
$$
a\cdot(z_1,\ldots, z_r,w_1,\ldots, w_s)=(z_1+\sigma_1(a),\ldots, z_r+\sigma_r(a), w_1+\sigma_{r+1}(a),\ldots, w_s+\sigma_{r+s}(a))\,, \quad \mbox{for all } a\in \mathcal{O}_{\mathbb{K}}
$$
and 
$$
u\cdot(z_1,\ldots, z_r,w_1,\ldots, w_s)=(\sigma_1(u)z_1,\ldots, \sigma_r(u)z_r, \sigma_{r+1}(u)w_1,\ldots, \sigma_{r+s}(u) w_s)\, , \quad \mbox{for every } u\in \mathcal{O}_{\mathbb{K}}^{*,+}\, .
$$
There always exists a free subgroup $U$ of rank $r$ of $\mathcal{O}_{\mathbb{K}}^{*,+}$ such that ${\rm pr}_{\mathbb{R}^r}\circ l(U)$ is a lattice of rank $r$ in $\mathbb{R}^r$, where
$l\colon \mathcal{O}_{\mathbb{K}}^{*,+}\to \mathbb{R}^{r+s}$
is the logarithmic representation of units 
$$
l(u)=(\log \sigma_1(u),\ldots, \log\sigma_r(u),2\log\lvert\sigma_{r+1}(u)\rvert,\ldots, 2\log\lvert\sigma_{r+s}(u)\rvert)
$$
and ${\rm pr}_{\mathbb{R}^r}\colon \mathbb{R}^{r+s}\to \mathbb{R}^r$ is the projection on the first $r$ coordinates. 
The action of $U\ltimes\mathcal{O}_{\mathbb{K}}$ on $\mathbb H^r\times \mathbb C^s$ is free, properly discontinuous  
and co-compact. An {\em Oeljeklaus-Toma manifold} is then defined as the quotient 
$$
M:=\frac{\mathbb H^r\times \mathbb C^s}{U\ltimes\mathcal{O}_{\mathbb{K}}}
$$
and it is a compact complex manifold having complex dimension $r+s$.  

The structure of torus bundle of an Oeljeklaus-Toma manifold can be seen as follows:  \\
we have
$$
\frac{\mathbb H^r\times \mathbb C^s}{ \mathcal O_{\mathbb{K}}}=\mathbb R_+^{r}\times \mathbb T^{r+2s}
$$ 
and that the action of $U$ on $\mathbb H^r\times \mathbb C^s$ induces an action on  $\mathbb R_+^{r}\times \mathbb T^{r+2s}$
such that, for every $x\in \mathbb \R^r_+$ and $u\in U$, the induced map 
$$
u\colon (x,\mathbb T^{r+2s})\mapsto (\sigma_1(u)x_1,\ldots, \sigma_r(u)x_r,\mathbb T^{r+2s}) 
$$
is a diffeomorphism. Hence  
$$
M=\frac{\mathbb R_+^{r}\times \mathbb T^{r+2s}}{U}
$$
inherits the structure of a $\mathbb T^{r+2s}$-bundle over $\mathbb T^r$.  
We denote by $\pi$ and $F$ the projections 
$$
\pi\colon \mathbb H^r\times \mathbb C^s\to M\,,\quad F\colon  M\to\mathbb T^r\,. 
$$

From the viewpoint of Lie groups, the universal covering of an Oeljeklaus-Toma manifold $M$ has a natural structure of solvable 
Lie group $G$ and the complex structure on $M$ lifts to a left-invariant complex structure \cite{Kas}. Therefore, Oeljeklaus-Toma manifolds can be seen as compact solvmanifolds with a left-invariant complex structure. The solvable structure on the universal covering of $M$ can be described in terms of the existence of a left-invariant 
$(1,0)$-coframe $\{\omega^1,\dots,\omega^r,\gamma^1,\dots,\gamma^s\}$ such that 
\begin{equation}
\label{eqsstr}
\begin{cases}
d\omega^k=\frac{\sqrt{-1}}{2}\omega^k\wedge\bar \omega^k\quad & k=1,\dots,r\,,\\
d\gamma^i=\sum_{k=1}^r\lambda_{ki}\,\omega^k\wedge \gamma^{i}-\sum_{k=1}^r\,\lambda_{ki}\bar{\omega}^{k}\wedge\gamma^i\quad & i=1,\dots, s\,,
\end{cases}
\end{equation}
where 
$$
\lambda_{ki}=\frac{\sqrt{-1}}{4}b_{ki}-\frac{1}{2}c_{ki} 
$$
and $b_{ki},c_{ki}\in \R$ depend on the embeddings $\sigma_j$ as  
\begin{equation}\label{sigmar}
\sigma_{r+i}(u)=\left(\prod_{k=1}^r(\sigma_k(u))^{\frac{b_{ki}}{2}}\right)e^{\sqrt{-1}\sum_{k=1}^r c_{ki} \log\sigma_k(u) }\,, 
\end{equation}
for any $u\in U$, $k=1,\dots, r$ and $i=1,\dots, s$. Since $U\subseteq\mathcal{O}_{\mathbb{K}}^*$,  it is easy to see that 
$$
l(U)\subseteq\left\{x\in \R^{r+s}\quad \middle|\quad \sum_{i=1}^{r+s}x_i=0\right\}.
$$ This fact together with \eqref{sigmar} implies that, for every $u\in U$, 
$$
\sum_{i=1}^r\log\sigma_i(u)\left(1+\sum_{k=1}^sb_{ik}\right)=0\, ,
$$
which, since  ${\rm pr}_{\mathbb{R}^r}\circ l(U)$ is a lattice of rank $r$ in $\mathbb{R}^r$, is equivalent to 
\begin{equation}\label{bkappai}
\sum_{k=1}^sb_{ik}=-1\,,\quad \mbox{for all } i=1,\ldots, r\,.
\end{equation}

The dual frame $\{Z_1,\dots, Z_r,W_1,\dots,W_s\}$ to $\{\omega^1,\dots,\omega^r,\gamma^1,\dots,\gamma^s\}$ satisfies the following structure equations:  
$$
[Z_k,\bar Z_k]=\,-\frac{\sqrt{-1}}{2}(Z_k+\bar Z_k)\,,\quad 
[Z_k,W_i]=-\lambda_{ki}W_i\,,\quad 
[ Z_k,\bar  W_i]=\bar \lambda_{ki} \bar W_i\,,
$$
for $k=1,\dots, r$, $i=1,\dots, s$. Consequently the Lie algebra $\g$ of the universal covering of $M$ splits as vector space as 
$$
\mathfrak{g}=\mathfrak{h}\oplus \mathfrak{I}
$$
where $\mathfrak{I}$ is an abelian ideal  and $\mathfrak{h}$ is a subalgebra isomorphic to $\underbrace{\mathfrak{f}\oplus \dots \oplus  \mathfrak{f}}_{\mbox{$r$-times}}$, where $\mathfrak f$ is the {\em filiform} Lie algebra $\mathfrak{f}=\langle e_1,e_2\rangle$, $[e_1,e_2]=-\tfrac12 e_1$. The complex structure $J$ induced on $\mathfrak g$ preserves both 
$\mathfrak h$ and $\mathfrak I$ and its restriction $J_{\mathfrak h}$ on $\mathfrak h$ satisfies 
$$
J_\mathfrak{h}=\underbrace{J_{\mathfrak f}\oplus \dots \oplus J_{\mathfrak f}}_{\mbox{$r$-times}}\, ,
$$
where $J_{\mathfrak f}$ is the complex structure on $\mathfrak f$ defined by $J_{\mathfrak f}(e_1)= e_2$. Moreover 
$$
[\mathfrak h^{1,0},\mathfrak I^{0,1}]\subseteq \mathfrak I^{0,1}\,.
$$

\section{Convergence in the Gromov-Hausdorff  sense}
We briefly recall Gromov-Hausdorff   convergence of metric spaces. The {\em Gromov-Hausdorff distance}  between two metric spaces $(X,d_X)$, $(Y,d_Y)$ is the infimum of all positive $\epsilon$ for which there exist two functions $F\colon X\to Y$, $G\colon Y\to X$, not necessarily continuous, satisfying the following four properties
$$
\begin{array}{cc}
|d_X(x_1,x_2)-d_Y(F(x_1),F(x_2))|\leq \epsilon\,, &  d_X(x,G(F(x)))\leq \epsilon\,,\\
|d_Y(y_1,y_2)-d_X(G(y_1),G(y_2))|\leq \epsilon\,, &  d_Y(y,F(G(y)))\leq \epsilon\,,
\end{array}
$$
for all $x,x_1,x_2\in X$ and $y,y_1,y_2\in Y$. If  $\{d_t\}_{t\in[0,\infty)}$ is a $1$-parameter family of distances on $X$, $(X,d_t)$ {{\em converges to $(Y,d_Y)$ in the Gromov-Hausdorff sense} if the Gromov-Hausdorff distance between $(X,d_t)$ and $(Y,d)$ tends to $0$ as $t\to \infty$.

\medskip 
Let $\{\omega_t\}_{t\in [0,\infty)}$ be a smooth curve of Hermitian metrics on an Oeljeklaus-Toma manifold and let $d_t$ be the induced distance on $M$. For a smooth curve $\gamma$ on $M$, let $L_t(\gamma)$ be   the length of $\gamma$ with respect to $\omega_t$. We further denote by $\mathcal H$ the foliation induced by $\mathfrak h$ on $M$. 

\begin{prop}\label{GH}
Let $\{\omega_t\}_{t\in [0,\infty)}$ be a smooth curve of Hermitian metrics on an Oeljeklaus-Toma manifold such that 
$$
\lim_{t\to \infty}\omega_t=\omega_{\infty}
$$
pointwise. 
Assume that there exist $T\in (0,\infty)$ and $C>0$ such that 
\begin{enumerate}
\item[1.] $L_t(\gamma)\leq C L_0(\gamma)\,,$ for every smooth curve $\gamma$ in $M$;

\vspace{0.1cm}
\item[2.] $L_t(\gamma)\leq (C/\sqrt{t})L_0(\gamma)$, for every smooth curve $\gamma$ in $M$ such that $\dot \gamma\in \ker \omega_\infty$.
\end{enumerate}
Assume further 
 \begin{enumerate}
\item[3.] for every $\epsilon,\ell >0$, there exists $T>0$ such that $|L_t(\gamma)-L_{\infty}(\gamma)|<\epsilon$, for every $t>T$ and every curve $\gamma$ in $M$ tangent to $\mathcal H$ and such that $L_\infty(\gamma)<\ell $. 
\end{enumerate}
Then $(M,d_t)$ converges in the Gromov-Hausdorff sense to $(\mathbb T^{r},d)$, where $d$ is the distance induced by  $\omega_{\infty}$ onto $\mathbb T^{r}$. 
\end{prop}
\begin{proof}
We follow the approach in \cite[Section 5]{TWCom} and in \cite[Proof of Theorem 1.1]{Zheng}. Let $M$ be an Oeljeklaus-Toma manifold. Consider the structure of $M$ as $\mathbb T^{r+2s}$-bundle over a  $\mathbb T^{r}$. Let $F\colon M\to \mathbb T^{r}$ be the projection onto the base and let $G\colon \mathbb T^{r}\to M$ be an arbitrary map such that  $F\circ G={\rm Id}_{\mathbb T^{r}}$.  We show that, for every $\epsilon>0$, there exists $T>0$ such that 
\begin{eqnarray}
&& |d_t(p,q)-d(F(p),F(q))|\leq \epsilon\,, \label{GH1} \\ 
&& |d(a,b)-d_t(G(a),G(b))|\leq \epsilon\,, \label{GH2} \\
&&  d_t(p,G(F(p)))\leq \epsilon\,, \label{GH3} \\
&& d(a,F(G(a)))\leq \epsilon\,, \label{GH4}
\end{eqnarray}
for every $t\geq T$, $p,q\in M$, $a,b\in \mathbb T^r$ which implies the statement.

Note that \eqref{GH4} is trivial since 
$$
d(a,F(G(a)))=0\,,
$$ 
for every $a\in \mathbb T^r$.

Then,  we show that \eqref{GH3} is satisfied. Let  $p,q\in M$  be two points in the same fiber over $\mathbb T^r$. Assume $p=\pi(z,w)$. We denote with $\mathcal{L}_{(z,w)}$ the leaf of the foliation $\ker\omega_{\infty}$ on the universal covering of $M$ passing through $(z,w)$. We easily see that,  for all $(z,w)\in\mathbb H^r\times \mathbb C^s$, $\mathcal{L}_{(z,w)}=\{z\}\times\mathbb C^s$.
In view of \cite[Section 2]{sima}, for every $z\in \mathbb H^r$, 
 $\pi(\{z\}\times \mathbb C^s)$ is  the leaf of the foliation $\ker\omega_{\infty}$ on $M$ passing through $p$ and it is dense in the fiber  $F^{-1}(F(p))$.
 Let $B_{R}$ be  the standard ball  in $\mathbb C^s$ about the origin having radius $R$. We can choose $R$ so that 
 every point in $F^{-1}(F(p))$ has distance with respect to $d_0$ less than $\epsilon/2C$ to $\pi(\{z\}\times \bar B_{R})$. On the other hand, given two points in $\pi(\{z\}\times \bar B_{R})$, they can be joined with a curve $\gamma$ in $F^{-1}(F(p))$ which is tangent to $\ker \omega_\infty$. Hence, for any such curve,  condition 2. implies 
 $$
 L_t(\gamma)\leq \frac{C'}{\sqrt t}\,,
 $$
 for a uniform constant $C'$ depending only on $R$. Let $p_0=\pi(z,0)$, $\gamma_1$ be a curve in $F^{-1}(F(p))$ connecting $p$ with $p_0$ tangent to $\ker\omega_\infty$ and  
 $\gamma_2$ be a curve connecting $p_0$ with $q$ 
 having minimal length with respect to $d_0$.  Hence, by using 1., for $t$ sufficiently large, we have 
 $$
 d_t(p,q)\leq L_{t}(\gamma_1)+L_t(\gamma_2)\leq \frac{C'}{\sqrt{t}}+CL_{0}(\gamma_2)\leq \frac{C'}{\sqrt{t}}+\frac{\epsilon}{2}\leq \epsilon\,,
$$ 
i.e. 
$$
d_t(p,q)\leq \epsilon 
$$ 
and  \eqref{GH3} follows.  

\medskip 
Next we show \eqref{GH1} and \eqref{GH2}. First of all,   we denote with $g$ the riemannian metric on $\mathbb{T}^r$ induced by $\omega_{\infty}$, for an explicit expression of $g$ see \cite[Section 2]{Zheng},  and we observe that  
\begin{equation}\label{Lg}
L_{g}(F(\gamma))\leq L_{\infty}(\gamma)\,,\mbox{ for every curve $\gamma$ in $M$,}
\end{equation}
and the equality holds if and only if
$$
\dot \gamma\in\mathcal{Y}={\rm span}_{\C}\left\{ \frac{1}{2\sqrt{-1}}\left(Z_i-\bar Z_i\right)\quad \middle | \quad i=1,\ldots, r \right\} .
$$

Let $p,q\in M$. We can find a curve $\gamma$ in $M$ connecting $p$ with a point $\tilde{q}$ in the $\mathbb T^{r+2s}$-fiber containing $q$ which is    
tangent to $\mathcal Y$  and such that $F(\gamma)$ is a minimal geodesic on $(\mathbb T^r,g)$, see for instance \cite[Proof of Theorem 5.1]{TWCom} or \cite[Proof of Theorem 1.1]{Zheng}.   By applying 3. we have 
 
$$
d_t(p,q)\leq d_t(p,\tilde q)+d_t(\tilde q,q)\leq d_t(p,\tilde q)+\epsilon\leq L_{t}( \gamma)+\epsilon\leq L_{\infty}(\gamma) +2\epsilon= L_g(F(\gamma))+ 2\epsilon=d(F(p),F(q))+2\epsilon\,, 
$$
for $t$ big enough, i.e. 
\begin{equation}\label{pippo}
d_t(p,q)-d(F(p),F(q))\leq \,2\epsilon\,,
\end{equation}
for $t$ sufficiently large. 

Next, using again \eqref{Lg},
we obtain, for $p,q\in M$, 
$$
d(F(p),F(q))\leq L_g(F(\gamma))\le L_{\infty}(\gamma)\leq L_{t}(\gamma)+\epsilon=d_t(p,q)+\epsilon\,,
$$
for $t$ big enough, 
where $\gamma$ is curve which realizes the  distance $d_t(p, q)$. Hence we obtain 
\begin{equation}\label{franco}
 d(F(p),F(q))-d_t(p,q)\le \epsilon\,.
\end{equation}
By substituting $p=G(a)$ and $q=G(b)$ in \eqref{pippo} and \eqref{franco} we infer 
$$
-\epsilon\leq d_t(G(a),G(b))-d(a,b)\leq 2\epsilon
$$
and \eqref{GH1} and \eqref{GH2} follow. 
\end{proof}

\section{The left-invariant Chern-Ricci flow on Oeljeklaus-Toma manifolds}\label{Sec3}
Given a Hermitian manifold $(M,\omega)$, the Chern connection of $\omega$ is the unique connection $\nabla$ on $(M,\omega)$ preserving both $\omega$ and the complex structure such that the $(1,1)$-component of its torsion tensor is vanishing.  The {\em Chern-Ricci form } of $\omega$ is the real closed $(1,1)$-form
$$
\rho_C(X,Y):=\sqrt{-1}\sum_{i=1}^n R_C(X,Y,X_i,\bar X_i)\,,
$$
where $R_C$ is the curvature tensor of $\nabla$ and $\{X_i\}$ is a unitary frame of $\omega$. The {\em Chern-Ricci flow} is then defined as the geometric flow 
$$
\partial_t\omega_t=-\rho_C(\omega_t)\,,\quad \omega_{|t=0}=\omega\,.
$$ 

In this section we prove the following 
\begin{prop}\label{main1}
Let $\omega$ be a left-invariant Hermitian metric on an Oeljeklaus-Toma manifold $M$. Then $\omega$ lifts to an expanding algebraic 
soliton for the Chern-Ricci flow on the universal covering of $M$ if and only if it takes the following expression with respect to the coframe $\{\omega^1,\dots,\omega^r,\gamma^1,\dots,\gamma^s\}$ satisfying \eqref{eqsstr}:
\begin{equation}\label{CRS}
\omega=\sqrt{-1}\left(A\sum_{i=1}^r\omega^i\wedge\bar \omega^i+\sum_{i,j=1}^sg_{r+i\overline{r+j}}\gamma^i\wedge\bar\gamma^j\right)\,.
\end{equation}
 Moreover, the Chern-Ricci flow starting from $\omega$ has a long-time solution $\{\omega_t\}$ such that $(M,
\frac{\omega_t}{1+t})$ converges as $t\to \infty$ in the Gromov-Hausdorff sense to $(\mathbb T^r,d)$, where $d$ is the distance induced by  $\omega_{\infty}$ onto $\mathbb T^{r}$. Finally, $(\mathbb H^r\times \C^s, \frac{\omega_t}{1+t})$ converges in the Cheeger-Gromov sense to $(\mathbb H^r\times \C^s, \tilde{\omega}_{\infty})$ where $\tilde{\omega}_{\infty}$ is an algebraic soliton.
\end{prop}

The proof of Proposition \ref{main1} is based on the following Theorem of Lauret 

\begin{theorem}[Lauret \cite{lauret}] \label{Jorge}
Let $(G,J)$ be a Lie group with a left-invariant complex structure. Then the Chern-Ricci form of a left-invariant Hermitian metric $\omega$ on $(G,J)$ does not depend on the Hermitian metric. Moreover, if $P\ne 0 $ is the endomorphism associated to $\rho_C$ with respect to  $\omega$, then the following are equivalent:\begin{enumerate}
\item\label{CRS1} $\omega$ is an algebraic soliton of the Chern-Ricci flow,
\item\label{CRS2} $P=cI + D$, for some $D\in {\rm Der}(\mathfrak{g}),$
\item\label{CRS3} The eigenvalues of 
$P$ are either $0$ or $c$, for some $c\in \R$ with $c\neq 0$, $\ker P$ is an abelian ideal of the Lie algebra of $G$ and   $(\ker P)^\perp$ is a subalgebra.
\end{enumerate}

\end{theorem}

\begin{proof}[Proof of Proposition $\ref{main1}$]
Let $M$ be an Oeljeklaus-Toma manifold. Since the Chern-Ricci form does not depend on  the choice of the left-invariant Hermitian metric, 
it is enough to compute $\rho_C$ for the \lq\lq canonical metric\rq\rq  
\begin{equation}\label{omegacan}
\omega=\sqrt{-1}\left(\sum_{i=1}^r\omega^i\wedge\bar\omega^i+\sum_{j=1}^s\gamma^j\wedge\bar \gamma^j\right)\,.
\end{equation}

We recall that the Chern-Ricci form of a left-invariant Hermitian metric $\omega=\sqrt{-1}\sum_{a=1}^n \alpha^a\wedge \bar \alpha^a $ on a Lie group $G^{2n}$ with a left-invariant complex structure
takes the following algebraic expression:
\begin{equation}\label{rhoC}
\rho_C(X,Y) = -\sum_{a=1}^n  (\omega([[X,Y]^{0,1},X_a],\bar X_a) +  \omega ([[X,Y]^{1,0},\bar X_a],X_a ) ) \, ,
\end{equation}
for every left-invariant vector fields $X,Y$ on $G$, where $\{\alpha^i\}$ is a left-invariant unitary $(1,0)$-coframe with dual frame   $\{X_a\}$ (see e.g. \cite{luigiproc}). By applying \eqref{rhoC} to  the canonical metric \eqref{omegacan} we have 
\begin{multline*}
\rho_C(X,Y) =-  \sum_{a=1}^r \{   \omega([[X,Y]^{0,1},Z_a],\bar Z_a) +  \omega([[X,Y]^{1,0},\bar Z_a],Z_a ) \}\\
-\,  \sum_{b=1}^s \{   \omega([[X,Y]^{0,1},W_b],\bar W_b) +\omega([[X,Y]^{1,0},\bar W_b],W_b ) \}\,.
\end{multline*}
Clearly, 																						
$$
\rho_C(Z_i,\bar Z_j)=0\, , \quad \mbox{ for all } i\ne j\,, \quad \rho_{C}(W_i,\bar W_j)=0\,, \quad \mbox{for every }i,j=1,\ldots, s\, .
$$
 Moreover,  since $\mathfrak{J}$ is an abelian ideal and $\omega$ makes $\mathfrak{J}$ and  $\mathfrak{h}$  orthogonal, we have:
 $$
 \rho_C(Z_i,\bar W_j)=0\,, \quad \mbox{for all } i=1,\ldots, r\,,\quad  j=1,\ldots, s\,.
 $$
Moreover we have 
$$
 \omega([[Z_i,\bar Z_i]^{0,1},Z_a],\bar Z_a)= \frac{\sqrt{-1}}{4}\delta_{ia}\, ,\quad 
\omega([[Z_i,\bar Z_i]^{1,0},\bar Z_a], Z_a)=\frac{\sqrt{-1}}{4}\delta_{ia}
 $$
and
 $$
 \omega([[Z_i,\bar Z_i]^{0,1},W_b],\bar W_b)= \frac12\lambda_{ib}\,, \quad 
 \omega([[Z_i,\bar Z_i]^{1,0},\bar W_b],W_b )= -\frac12\bar\lambda_{ib} 
$$
which imply
$$
\rho_C(Z_i,\bar Z_i)=-\sqrt{-1}\left(\frac12+\sum_{b=1}^s\Im\mathfrak{m}(\lambda_{ib})\right)=-\frac{\sqrt{-1}}{4}\,.
$$
and, consequently, 
$$
\rho_C=-\omega_\infty\,.
$$
In general, we have that 
$$
P_i^j=(\rho_C)_{i\bar k }g^{\bar k j }=\begin{cases}-\frac14g^{\bar i j } &\quad \mbox{if }i\in\{1,\ldots, r\}\,,\\
0 &\quad \mbox{otherwise }\,.
\end{cases}
$$
 Then, part \eqref{CRS3} of Theorem \ref{Jorge} readily implies that any left-invariant Hermitian metrics of the form \eqref{CRS} lifts to an expanding algebraic soliton on the universal covering of $M$ with cosmological constant $c=\frac{1}{4A}$. Conversely, let $\omega$ be an algebraic soliton for the Chern-Ricci flow. Then, thanks to part  \eqref{CRS2} of Theorem \ref{Jorge}, we have that 
 $$
 P-cI\in{\rm Der}(\g)\,.
 $$
On the other hand, we can easily see that, if $ D\in{\rm Der}(\g)$, then $\mathfrak h \subseteq\ker D$, see proof of Corollary \ref{Cor 5.4} in the present paper for the details. This readily implies that 
$$
-\frac14g^{i\bar i }=-\frac14g^{\bar j j}=c\,, \quad \mbox{ for all }i,j=1,\ldots, r\,, \quad g^{\bar i j }=0\,, \quad \mbox{for all }i \in\{1,\ldots, r\}\,,\, j\ne  i \,,
$$
from which the claim follows.

Moreover, the Chern-Ricci flow evolves an arbitrary left-invariant Hermitian metric $\omega$ as $\omega_t=\omega+t\omega_\infty$
and $\frac{\omega_t}{1+t}\to\omega_\infty$ as  $t\to \infty$.
In order to obtain  the claim regarding the Gromov-Hausdorff convergence, we show that $\frac{\omega_t}{1+t}$ satisfies conditions 1,2,3 in Proposition \ref{GH}. Here we 
denote by $|\cdot|_t$ the norm induced by $\omega_t$.  

\smallskip
Condition 2 is trivially satisfied since $\omega_{t|\mathfrak I\oplus \mathfrak I}=\omega_0$, for every $t\geq 0$,  and 
$$
L_t(\gamma)=\frac{1}{\sqrt{1+t}}L_{0}(\gamma)\,,
$$
for every curve $\gamma$ in $M$ tangent to $\ker\omega_\infty$. 

\smallskip 
On the other hand,  for a vector $v\in \mathfrak h$, we have 
$$
\frac{1}{\sqrt{1+t}}|v|_t\leq C|v|_{0}\,,  
$$
for a constant $C>0$ independent on $v$. This, together with condition 2, guarantees  condition 1. 

\smallskip

In order to prove condition 3, let $\epsilon, \ell >0$ and $T>0$ be such that 
$$
\left\vert \frac{|v|_t}{\sqrt{1+t}}-|v|_\infty\right\vert\leq \frac{\epsilon}{\ell}\,, 
$$
for every $v\in \mathfrak h$ and $t\geq T$. Let $\gamma$ be a curve in $M$ tangent to $\mathcal H$ which is parametrized by arclength with respect to $\omega_\infty$ and such that $L_{\infty}(\gamma)<\ell$. 
Then
$$
|L_{t}(\gamma)-L_{\infty}(\gamma)|\leq \int_{0}^{b}\left\vert\frac{1}{\sqrt{1+t}}|\dot \gamma |_t-|\dot \gamma|_\infty\right\vert da\leq \frac{\epsilon}{\ell}b\leq \epsilon\,, 
$$
since $b\leq \ell$. 

For the last statement,  we identify  $\omega_t$ with its pull-back onto $\mathbb H^r\times\mathbb C^s$ and  we fix  as base point the identity element of $\mathbb H^r\times\mathbb C^s$. Firstly,  we  observe that the endomorphism $D$ represented with respect to the frame $\{Z_1,\ldots, Z_r, W_1,\ldots, W_s\}$ by the following matrix:
$$
\begin{pmatrix}
0&0\\0& {\rm I}_{\mathfrak J}
\end{pmatrix}
$$ 
is a derivation of $\g$. Moreover, we can construct 
$$
\exp(s(t)D)=\begin{pmatrix}{\rm I}_{\mathfrak{h}}& 0 \\ 0 & e^{s(t)}{\rm I}_{\mathfrak{J}}\end{pmatrix}\in {\rm Aut}(\g, J)\,, \quad \mbox{for every } t\ge 0\,,
$$
where $s(t)=\log(\sqrt{1+t})$  and define the 1-parameter family $\{\varphi_t\}\subseteq{\rm Aut}(\mathbb H^r\times\C ^s, J )$ such that 
$$
d\varphi_t=\exp(s(t)D)\,, \quad \mbox{for every } t\ge 0\, .
$$
Trivially, we see that 
$$
\begin{aligned}
\varphi_t^*\frac{\omega_t}{1+t}(Z_i,\bar Z_j)=&\, \sqrt{-1}\frac{1}{1+t}\left(g_{i\bar j }+\frac t4\delta_{ij}\right )\to \frac{\sqrt{-1}}{4}\delta_{ij} \quad \mbox{as }t\to\infty\,,\\
\varphi_t^*\frac{\omega_t}{1+t}(Z_i,\bar W_j)=&\,\sqrt{-1}\frac{e^{s(t)}}{1+t}g_{i\overline{r+j}}\to  0 \quad \mbox{as }t\to\infty\,,\\
\varphi_t^*\frac{\omega_t}{1+t}(W_i,\bar W_j)=&\,\sqrt{-1}\frac{e^{2s(t)}}{1+t}g_{r+i\overline{r+j}} \to \sqrt{-1}g_{r+i\overline{r+j}} \quad \mbox{as }t\to\infty\,.\\
\end{aligned}
$$
These facts guarantee that 
$$
\varphi_t^*\frac{\omega_t}{1+t}\to \omega_{\infty}+\omega_{\mathfrak{J}\oplus \mathfrak{J}} \quad \mbox{as }t\to\infty\,, 
$$
hence, the assertion follows.
\end{proof}

\section{Proof of the main result}\label{Sec4}
In this section we prove Theorem \ref{main2}.

\medskip 
The 
existence of pluriclosed metrics on Oeljeklaus-Toma manifolds was studied in \cite{A}, \cite{FKV} and \cite{Otiman}. In particular from \cite{A} it follows the following result.  
\begin{theorem}[Corollary 3, \cite{A}]
An Oeljeklaus-Toma manifold of type $(r,s)$ admits a pluriclosed metric if and only if $r=s$ and 
\begin{equation}\label{onsigma}
\sigma_j(u)|\sigma_{r+j}(u)|^2=1\,,\quad  \mbox{ for every } j=1,\dots,s \mbox{ and }u\in U\,. 
\end{equation}
\end{theorem}

Condition \eqref{onsigma} in the previous Theorem can be rewritten in terms of the structure constants appearing in \eqref{eqsstr}. Indeed,  \eqref{eqsstr} together with \eqref{onsigma} forces 
 $b_{ki}\in \{0,-1\}$  and $b_{ki}b_{li}=0$,  for every $i,k,l=1,\dots, s$ with 
$k\neq l$. In particular, using \eqref{bkappai}, for every fixed index  $k\in\{1,\ldots, s\}$, there exists a unique  $i_k\in \{1,\ldots, s\}$ such that 
$$
b_{ki_k}=-1\,,\quad b_{ki}=0\,,
$$
for all $i\ne i_k$ and, if  $k\ne l ,$ then $i_k\ne i_l$. Hence, up to a reorder of the $\gamma_j$'s, we may and do assume,  without loss of generality, $i_k=k$,  for every $k\in\{1,\dots,s\}$,
i.e.  
\begin{equation}
\label{lambdas}
 \lambda_{ki}=\begin{cases}-\frac12 c_{ki} \quad &\mbox {if } i\ne k\,, \\ 
 -\frac12 c_{kk}-\frac{\sqrt{-1}}{4}\quad & \mbox{if } i=k\,.\end{cases}
 \end{equation}

\begin{prop}[Characterization of left-invariant pluriclosed metrics on  Oeljeklaus-Toma manifolds]\label{ch}
A left-invariant metric $\omega$ on  an Oeljeklaus-Toma manifold admitting pluriclosed metrics is pluriclosed if and only if  it  takes the following expression with respect to a coframe $\{\omega^1,\dots,\omega^s,\gamma^1,\dots,\gamma^s\}$ satisfying \eqref{eqsstr} and \eqref{lambdas}:
\begin{equation}\label{gPC}
\omega=\sqrt{-1}\sum_{i=1}^sA_i\omega^i\wedge \bar \omega^i+ B_i\gamma^i\wedge \bar \gamma^i+ \sqrt{-1}\sum_{r=1}^{k}\left( C_{r}\omega^{p_r}\wedge \bar \gamma^{p_r}+\bar C_{r}\gamma^{p_r}\wedge\bar \omega^{p_r} 
\right)
\end{equation}
for some $A_1,\dots,A_s,B_1,\dots,B_s\in \mathbb R_+$,  $C_1,\dots,C_k\in 	\mathbb C$, where $\{p_1,\dots,p_k\}\subseteq \{1,\dots,s\}$ are such that 
$$
\lambda_{jp_i}=0\, ,\mbox{ for all }j\neq p_i\, , \mbox{ for all } i=1,\dots,k\,. 
$$ 
\end{prop}
\begin{proof}
We assume $s>1$ since the case $s=1$ is trivial.
Let
   \[\omega=\sqrt{-1}\sum_{p,q=1}^sA_{p\bar q}\omega^p\wedge\bar{\omega}^q+B_{p\bar q}\gamma^p\wedge\bar{\gamma}^q+C_{p\bar q}\omega^p\wedge\bar{\gamma}^q+\bar C_{ p\bar q} \gamma^q\wedge \bar{\omega}^p\]
be an arbitrary real left-invariant $(1,1)$-form on $M$, with $A_{p\bar p}, B_{p\bar p}\in \R$, for every $p=1,\ldots, s$,  $A_{p\bar q}, B_{p\bar q}\in \C$, for all $p,q=1,\ldots, s$ with $ p\ne q$, and $ C_{p\bar q}\in \C$,  for every $p,q=1,\ldots, s$. 

From the structure equations \eqref{eqsstr},  it easily follows 
\begin{equation}\label{3}
\begin{cases}
& \partial\bar\partial(\omega^p\wedge\bar\omega^q)\in \langle \omega^{p}\wedge \omega^p\wedge \bar \omega^p\wedge \bar\omega ^q \rangle \\ 
& \partial\bar\partial(\omega^p\wedge\bar\gamma^q)\in \langle \omega^{i}\wedge \omega^j\wedge \bar \omega^l\wedge \bar\gamma^m\rangle \\ 
& \partial\bar\partial(\gamma^p\wedge\bar\gamma^q)\in \langle \omega^{i}\wedge\bar  \omega^j\wedge\gamma^l\wedge \bar\gamma^m\rangle \\ 
\end{cases}
\end{equation}
and that $\omega$ is pluriclosed if and only if the following three conditions are satisfied 
\begin{eqnarray}
&& \sum_{p,q=1}^sA_{p\bar q}\partial\bar{\partial}(\omega^p\wedge\bar{\omega}^q)=0 \label{111}\,;\\ 
&& \sum_{p,q=1}^sB_{p\bar q}\partial\bar{\partial}(\gamma^p\wedge\bar{\gamma}^q)=0 \label{222}\,;\\ 
&& \sum_{p,q=1}^sC_{p\bar q}\partial\bar{\partial}(\omega^p\wedge\bar{\gamma}^q)=0 \label{333}\,.
 \end{eqnarray}
The first relation in \eqref{3} yields that \eqref{111} is satisfied if and only if  
$$
  A_{p\bar q}=0\,, \mbox{ for all }p\neq q\,.
$$

Next we focus on \eqref{222}.
We have
$$
\begin{aligned}
\partial\bar\partial(\gamma^p\wedge\bar{\gamma}^q)=&
\partial\left(-\sum_{\delta=1}^s\lambda_{\delta p}\bar{\omega}^{\delta}\wedge\gamma^p\wedge\bar{\gamma}^q-\gamma^p\wedge\sum_{\delta=1}^s\bar\lambda_{\delta q}\bar{\omega}^{\delta}\wedge\bar{\gamma}^q\right)\,
\end{aligned}
$$
and 
$$
\begin{aligned}
\partial\bar\partial(\gamma^p\wedge\bar{\gamma}^q)=&\,
\sum_{\delta=1}^s(\bar\lambda_{\delta q}-\lambda_{\delta p})\left(\partial\bar{\omega}^{\delta}\wedge\gamma^p\wedge\bar{\gamma}^q-\bar{\omega}^{\delta}\wedge\partial\gamma^p\wedge\bar{\gamma}^q+\bar{\omega}^{\delta}\wedge \gamma^p\wedge \partial\bar{\gamma}^q\right)   
\end{aligned}
$$
which implies 
$$
\begin{aligned}
\partial\bar\partial(\gamma^p\wedge\bar{\gamma}^q)=&\, \sum_{\delta=1}^s\frac{\sqrt{-1}}{2}(\bar\lambda_{\delta q}-\lambda_{\delta p})\omega^{\delta}\wedge\bar{\omega}^{\delta}\wedge\gamma^p\wedge\bar{\gamma}^q-\sum_{\delta=1}^s(\bar\lambda_{\delta q}-\lambda_{\delta p})\bar{\omega}^{\delta}\wedge\left(\sum_{a=1}^s\lambda_{ap}\omega^a\wedge\gamma^p\right)\wedge\bar{\gamma}^q\\
&\,+\sum_{\delta=1}^s(\bar\lambda_{\delta q}-\lambda_{\delta p})\bar{\omega}^{\delta}\wedge\gamma^{p}\wedge\left(-\sum_{a=1}^s\bar\lambda_{aq}\omega^a\wedge\bar{\gamma}^q\right)\\
=&\sum_{\delta=1}^s\frac{\sqrt{-1}}{2}(\bar\lambda_{\delta q}-\lambda_{\delta p})\omega^{\delta}\wedge\bar{\omega}^{\delta}\wedge{\gamma^p}\wedge\bar{\gamma}^q+\sum_{\delta, a}(\lambda_{ap}-\bar\lambda_{aq})(\bar\lambda_{\delta q}-\lambda_{\delta p})\omega^a\wedge\bar{\omega}^{\delta}\wedge\gamma^p\wedge\bar{\gamma}^q\,.
\end{aligned}
$$
Finally, we get 
\begin{multline*}
\partial\bar\partial(\gamma^p\wedge\bar{\gamma}^q)=\sum_{\delta=1}^s(\bar\lambda_{\delta q}-\lambda_{\delta p})\left(\frac{\sqrt{-1}}{2}+\lambda_{\delta p}-\bar\lambda_{\delta q}\right)\omega^{\delta}\wedge\bar{\omega}^{\delta}\wedge{\gamma^p}\wedge\bar{\gamma}^q\\+\sum_{\delta\ne a}(\lambda_{ap}
-\bar\lambda_{aq})(\bar\lambda_{\delta q}-\lambda_{\delta p})\omega^a\wedge\bar{\omega}^{\delta}\wedge\gamma^p\wedge\bar{\gamma}^q
\end{multline*}
and that condition \eqref{222} is equivalent to
\begin{equation*}
B_{p\bar q}\left(\sum_{\delta=1}^s(\bar\lambda_{\delta q}-\lambda_{\delta p})\left(\frac{\sqrt{-1}}{2}
+\lambda_{\delta p}-\bar\lambda_{\delta q}\right)\omega^{\delta}\wedge\bar{\omega}^{\delta}
+\sum_{\delta\ne a}(\lambda_{ap}-\bar\lambda_{aq})(\bar\lambda_{\delta q}-\lambda_{\delta p})\omega^a\wedge\bar{\omega}^{\delta}\right)=0\,,
\end{equation*}
 for every $p,q=1,\dots, s\,.$

By using our conditions on the $b_{ki}$'s, it is easy to show that the quantity 
$$
\sum_{\delta=1}^s(\bar\lambda_{\delta q}-\lambda_{\delta p})\left(\frac{\sqrt{-1}}{2}
+\lambda_{\delta p}-\bar\lambda_{\delta q}\right)\omega^{\delta}\wedge\bar{\omega}^{\delta}
+\sum_{\delta\ne a}(\lambda_{ap}-\bar\lambda_{aq})(\bar\lambda_{\delta q}-\lambda_{\delta p})\omega^a\wedge\bar{\omega}^{\delta}
$$
is vanishing for $p=q$ and, consequently, there are no restrictions on the $B_{q\bar q}$'s. Now we observe that the real part of 
$$
(\bar\lambda_{p q}-\lambda_{p p})\left(\frac{\sqrt{-1}}{2}
+\lambda_{p p}-\bar\lambda_{p q}\right)
$$
is different from $0$, for every $p,q$ with $p\neq q$, which forces $B_{p\bar q}=0$, for $p\neq q$. Indeed, we have 
\begin{eqnarray*}
&&\bar \lambda_{\delta q}-\lambda_{\delta p}=\frac12 (c_{\delta p}-c_{\delta q})-\frac{\sqrt{-1}}{4}(b_{\delta p}+b_{\delta q})\,, \\
&& \frac{\sqrt{-1}}{2}+\lambda_{\delta p}-\bar\lambda_{\delta q}=-\frac12(c_{\delta p}-c_{\delta q})+\frac{\sqrt{-1}}{2}\left(1+\frac{b_{\delta p}+b_{\delta q}}{2}\right)\\
\end{eqnarray*}
which implies   
\begin{equation}\label{rotula}
\Re\mathfrak e\,\left((\bar \lambda_{\delta q}-\lambda_{\delta p})\left(\frac{\sqrt{-1}}{2}+\lambda_{\delta q}-\bar \lambda_{\delta p}\right)\right)=-\frac{(c_{\delta p}-c_{\delta q})^2}{4}+\frac14\left(\frac{b_{\delta p}+b_{\delta q}}{2}\right)\left(1+\frac{b_{\delta p}+b_{\delta q}}{2}\right)\,.
\end{equation}
Since  $p\ne q$, we have   
 $$
 b_{p p}=-1\,,\quad  b_{p q}=0\, ,
 $$ 
and so \eqref{rotula} computed for $\delta=q$ gives
$$
\Re\mathfrak e\left(\left(\bar \lambda_{p q}-\lambda_{p p}\right)\left(\frac{\sqrt{-1}}{2}+\lambda_{p q}-\bar \lambda_{p p})\right)\right)=\frac14\left( -(c_{p p}-c_{p q})^2-\frac14\right)\ne 0\,,
$$ 
as required. Therefore equation \eqref{222} is satisfied if and only if
$$
  B_{p\bar q}=0\,, \mbox{ for all }p\neq q\,.
$$

Next we focus on \eqref{333}. We have
$$
\begin{aligned}
\partial\bar{\partial}(\omega^p\wedge\bar{\gamma}^q)=&\,
\partial\left(\frac{\sqrt{-1}}{2}\omega^p\wedge\bar{\omega}^p\wedge\bar{\gamma}^q-\omega^p\wedge\left(\sum_{\delta=1}^s\bar \lambda_{\delta q}\bar{\omega}^{\delta}\wedge \bar{\gamma}^q\right)\right)\\
\end{aligned}
$$
and    
$$
\begin{aligned}
\partial\bar{\partial}(\omega^p\wedge\bar{\gamma}^q)=&\,
\frac{\sqrt{-1}}{2}\left(-\frac{\sqrt{-1}}{2}\omega^p\wedge\omega^p\wedge\bar{\omega}^p\wedge\bar{\gamma}^q+\omega^p\wedge\bar{\omega}^p\wedge\left(-\sum_{\delta=1}^s\bar {\lambda}_{\delta q}\omega^{\delta}\wedge\bar{\gamma}^q\right)\right)\\
&\,+\sum_{\delta=1}^s\frac{\sqrt{-1}}{2}\bar \lambda_{\delta q}\omega^p\wedge\omega^{\delta}\wedge\bar{\omega}^{\delta}\wedge\bar{\gamma}^q+\sum_{\delta=1}^s\bar \lambda_{\delta q}\omega^p\wedge\bar{\omega}^{\delta}\wedge\left(\sum_{a=1}^s\bar \lambda_{aq}\omega^a\wedge\bar{\gamma}^q\right)\,.
\end{aligned}
$$
Hence we get 
\begin{multline*}
\partial\bar{\partial}(\omega^p\wedge\bar{\gamma}^q)=
\sum_{\substack{\delta=1\\ \delta\ne p}}^s\frac{\sqrt{-1}}{2}\bar\lambda_{\delta q}\omega^p\wedge\bar{\omega}^p\wedge\omega^{\delta}\wedge\bar\gamma^q+\sum_{\substack{\delta=1\\ \delta\ne p}}^s\frac{\sqrt{-1}}{2}\bar \lambda_{\delta q}\omega^p\wedge\omega^{\delta}\wedge\bar{\omega}^{\delta}\wedge\bar{\gamma}^q\\ +\sum_{\substack{\delta, a \\ a\ne p}}\bar \lambda_{\delta q}\bar \lambda_{aq}\omega^p\wedge\bar{\omega}^{\delta}\wedge\omega^a\wedge\bar{\gamma}^q
\end{multline*}

and 
\begin{multline*}
\partial\bar{\partial}(\omega^p\wedge\bar{\gamma}^q)=\sum_{\substack{\delta=1\\ \delta\ne p}}^s\frac{\sqrt{-1}}{2}\bar\lambda_{\delta q}\omega^p\wedge\bar{\omega}^p\wedge\omega^{\delta}\wedge\bar\gamma^q+\sum_{\substack{a=1\\ a\ne p }}^s\bar\lambda_{pq}\bar\lambda_{aq}\omega^p\wedge\bar{\omega}^{p}\wedge\omega^a\wedge\bar{\gamma}^q\\+\sum_{\substack{\delta=1\\ \delta\ne p}}^s\frac{\sqrt{-1}}{2}\bar\lambda_{\delta q}\omega^p\wedge\omega^{\delta}\wedge\bar{\omega}^{\delta}\wedge\bar{\gamma}^q+\sum_{\substack{\delta, a \\ \delta\ne p \\ a\ne p}}\bar\lambda_{\delta q}\bar\lambda_{aq}\omega^p\wedge\bar{\omega}^{\delta}\wedge\omega^a\wedge\bar{\gamma}^q\,.
\end{multline*}

Therefore
\begin{multline*}
\partial\bar{\partial}(\omega^p\wedge\bar{\gamma}^q)=\sum_{\substack{\delta=1\\ \delta\ne p}}^s\bar\lambda_{\delta q}\left(\frac{\sqrt{-1}}{2}+\bar\lambda_{pq}\right)\omega^p\wedge\bar{\omega}^p\wedge\omega^{\delta}\wedge\bar\gamma^q+\\
\sum_{\substack{\delta=1\\ \delta\ne p}}^s\bar\lambda_{\delta q}\left(\frac{\sqrt{-1}}{2}-\bar\lambda_{\delta q}\right)\omega^p\wedge\omega^{\delta}\wedge\bar{\omega}^{\delta}\wedge\bar{\gamma}^q+\sum_{\substack{\delta\ne a \\ \delta\ne p \\ a\ne p}}\bar\lambda_{\delta q}\bar\lambda_{aq}\omega^p\wedge\bar{\omega}^{\delta}\wedge\omega^a\wedge\bar{\gamma}^q
\end{multline*}
and \eqref{333} is equivalent to 
$$
C_{p\bar{q}}\left(\sum_{\substack{\delta=1\\ \delta\ne p}}^s\bar\lambda_{\delta q}\left(\frac{\sqrt{-1}}{2}+\bar\lambda_{pq}\right)\bar{\omega}^p\wedge\omega^{\delta}
+\sum_{\substack{\delta=1\\ \delta\ne p}}^s\bar\lambda_{\delta q}\left(\frac{\sqrt{-1}}{2}-\bar\lambda_{\delta q}\right)\omega^{\delta}\wedge\bar{\omega}^{\delta}
+\sum_{\substack{\delta\ne a \\ \delta\ne p \\ a\ne p}}\bar\lambda_{\delta q}\bar\lambda_{aq}\bar{\omega}^{\delta}\wedge\omega^a\right)=0\,, 
$$
for every $p,q=1,\dots,s$. Since 
$$
\lambda_{pq}\neq \pm \frac{\sqrt{-1}}{2}\, ,\quad \mbox {for all }\,\,p,q=1,\dots,s\, ,
$$
the quantity 
$$
E_{p\bar q}:=\sum_{\substack{\delta=1\\ \delta\ne p}}^s\bar\lambda_{\delta q}\left(\frac{\sqrt{-1}}{2}+\bar\lambda_{pq}\right)\bar{\omega}^p\wedge\omega^{\delta}
+\sum_{\substack{\delta=1\\ \delta\ne p}}^s\bar\lambda_{\delta q}\left(\frac{\sqrt{-1}}{2}-\bar\lambda_{\delta q}\right)\omega^{\delta}\wedge\bar{\omega}^{\delta}
+\sum_{\substack{\delta\ne a \\ \delta\ne p \\ a\ne p}}\bar\lambda_{\delta q}\bar\lambda_{aq}\bar{\omega}^{\delta}\wedge\omega^a
$$
is vanishing if and only if    
$$
\lambda_{\delta q}=0\, , \quad \mbox{ for all } \delta\ne p\,.
$$
Since $\lambda_{qq}\neq 0$, it follows 
$$
E_{p\bar q}\neq 0\,,\quad \mbox{ for every $p,q$ with $p\neq q$}
$$
and 
$$
E_{p\bar p}=0\mbox{ if and only if $c_{\delta p}=0$\,, for all $\delta\neq p$}\, . 
$$

Hence the claim follows. 
\end{proof}

\begin{prop}\label{prop4.3}
Let 
\begin{equation}\label{palpebra}
\omega=\sqrt{-1}\sum_{i=1}^sA_i\omega^i\wedge \bar \omega^i+ B_i\gamma^i\wedge \bar \gamma^i+ \sqrt{-1}\sum_{r=1}^{k}\left( C_{r}\omega^{p_r}\wedge \bar \gamma^{p_r}+\bar C_{r}\gamma^{p_r}\wedge \bar \omega^{p_r} 
\right)
\end{equation}
be a left-invariant pluriclosed Hermitian metric 
on an Oeljeklaus-Toma manifold, where the components are with respect  to a coframe $\{\omega^1,\dots,\omega^s,\gamma^1,\dots,\gamma^s\}$ satisfying \eqref{eqsstr} and \eqref{lambdas}
and $\{p_1,\dots,p_k\}\subseteq \{1,\dots,s\}$ are such that 
$$
\lambda_{jp_i}=0\, ,\mbox{ for all }j\neq p_i\, , \mbox{ for all } i=1,\dots,k\,. 
$$ 
Then, the $(1,1)$-part of the Bismut-Ricci form of $\omega$ takes the following expression:
$$
\rho^{1,1}_B=-\sqrt{-1}\sum_{r=1}^k\frac{3}{4}\left(1+\frac{\lvert C_{r}\rvert^2}{A_{p_r}B_{p_r}-\lvert C_{r}\rvert^2}\right)\omega^{p_r}\wedge \bar \omega^{p_r}-\sqrt{-1}\sum_{i\not\in\{p_1,\ldots, p_k\}} \frac{3}{4}\omega^i\wedge\bar{\omega}^i 
$$
$$
-\sqrt{-1}\sum_{r=1}^k\left(-\frac{3}{16}-\frac{c_{p_rp_r}^2}{4}-\frac{\sqrt{-1}c_{p_rp_r}}{4}\right)\frac{B_{p_r}C_r}{A_{p_r}B_{p_r}-|C_r|^2}\omega^{p_r}\wedge \bar \gamma^{p_r} + \quad \mbox{{\em conjugates}}\, .
$$
\end{prop}
\begin{proof}
We recall that the Bismut-Ricci form of a left-invariant Hermitian metric $\omega=\sqrt{-1}\sum_{a,b=1}^n g_{a\bar b}\alpha^a\wedge \bar \alpha^b $ on a Lie group $G^{2n}$ with a left-invariant complex structure
takes the following algebraic expression: 
\begin{equation}\label{rhoB}
\rho_B(X, Y)=-\sum_{a,b=1}^{n}g^{a\bar b}\omega([[X,Y]^{1,0}, X_a], \bar{X_b})+ g^{\bar a b  }\omega([[X,Y]^{0,1},\bar{X_a}], X_b)+\sqrt{-1} \sum_{a,b=1}^ng^{a\bar b}\omega([X,Y],J[X_a,\bar{X_b}])\, ,
\end{equation}
for every left-invariant vector fields $X,Y$ on $G$, where $\{\alpha^i\}$ is a left-invariant $(1,0)$-coframe with dual frame   $\{X_a\}$ and $(g^{\bar ba})$ is the inverse matrix to $(g_{i\bar j})$  (see e.g. \cite{luigiproc}). We apply \eqref{rhoB} to a left-invariant Hermitian metric on an Oeljeklaus-Toma manifold of the form \eqref{palpebra}.

We have 
\begin{eqnarray*}
g^{\bar i s+i}=\begin{cases} 0 \quad & \mbox{if }i\not\in\{p_1,\ldots, p_k\}\, , \\ -\frac{C_i}{A_iB_i-\lvert C_i\rvert^2}\quad & \mbox{otherwise}\, , 
\end{cases}\quad  g^{\bar i i}=\frac{B_i}{A_iB_i-\lvert C_i\rvert^2}\,, \quad g^{\overline{ s+i}s+i }=\frac{A_i}{A_iB_i-\lvert C_i\rvert^2}
\end{eqnarray*}
and taking into account that the ideal  $\mathfrak{I}$ is abelian, we have 
$$
\rho_B(X,Y)=-\sum_{i=1}^4\rho_i(X,Y)\,,
$$
where 
\begin{eqnarray*}
&& \rho_1(X,Y)=\sum_{a=1}^sg^{a\bar a}(\omega([[X,Y]^{1,0}, Z_a], \bar Z_a)-\frac{ \sqrt{-1}}{2}\omega([X,Y],Z_a-\bar Z_a)+\omega([[X,Y]^{0,1},\bar Z_a],Z_a))\,,\\
&&\rho_2(X,Y)=\sum_{a=1}^sg^{s+a\overline {s+a}}(\omega([[X,Y]^{1,0}, W_a], \bar W_a)+\omega([[X,Y]^{0,1},\bar W_a], W_a))\,,\\
&&\rho_3(X,Y)=\sum_{r=1}^kg^{p_r\overline{s+p_r}}\left(\omega([[X,Y]^{1,0}, Z_{p_r}], \bar W_{p_r})-\omega([X,Y],[Z_{p_r},\bar W_{p_r}])\right)+g^{\overline {p_r} s+p_r  }\omega([[X,Y]^{0,1},\bar Z_{p_r}], W_{p_r})\, ,\\
&&\rho_4(X,Y)=\sum_{r=1}^kg^{s+p_r\bar p_r}\left(\omega([[X,Y]^{1,0}, W_{p_r}], \bar Z_{p_r})+\omega([X,Y],[W_{p_r},\bar Z_{p_r}]))\right)+ g^{\overline{s+p_r} p_r  }\omega([[X,Y]^{0,1},\bar W_{p_r}], Z_{p_r})\, .
\end{eqnarray*}

Next we focus on the computation of $\rho_B(Z_i,\bar Z_j).$ Thanks to \eqref{eqsstr}, we easily obtain that 
$$
\rho_B(Z_i,\bar Z_j)=0\, , \quad \mbox{for every } i,j=1,\ldots, s\, ,\,  \, i\ne j\, .
$$ 
On the other hand,  
$$
\begin{aligned}
\rho_1(Z_i,\bar Z_i)=&
 -\frac{\sqrt{-1}}{2}\sum_{a=1}^sg^{a\bar a}\left(-\frac{\sqrt{-1}}{2}\omega(Z_i+\bar Z_i, Z_a-\bar Z_a)\right)
= \,\frac{\sqrt{-1}}{2}g^{i\bar i}A_i=\frac{\sqrt{-1}}{2}\left(\frac{A_iB_i}{A_iB_i-\lvert C_i\rvert^2}\right).
\end{aligned}
$$
Moreover, we have 
$$
\begin{aligned}
\rho_2(Z_i,\bar Z_i)=&
-\frac{\sqrt{-1}}{2}\sum_{a=1}^sg^{s+a\overline {s+a}} (\omega([Z_i, W_a], \bar W_a)+\omega([\bar Z_i,\bar W_a],W_a)\\
=&-\sqrt{-1}\sum_{a=1}^sg^{s+a\overline {s+a}}\Re\mathfrak e\, \omega([Z_i,W_a], \bar W_a)\,.
\end{aligned}
$$
Using  \eqref{eqsstr}, we have 
$$
\omega([Z_i, W_a], \bar W_a)=
-\sqrt{-1}\lambda_{ia}B_a\, ,
$$
$$
\Re\mathfrak e\, \omega([Z_i,W_a], \bar W_a)=
\frac{B_ab_{ia}}{4}=-\frac{B_a}{4}\delta_{ia}\,.
$$
Then, 
$$
\rho_2(Z_i,\bar Z_i)=\sqrt{-1}\frac{g^{s+i\overline{s+i}}B_i}{4}=\frac{\sqrt{-1}}{4}\frac{A_iB_i}{A_i B_i-\lvert C_i\rvert^2}\,.
$$
Next we observe that 
$$
\rho_3(Z_i,\bar Z_i)+\rho_4(Z_i,\bar Z_i)=0
$$
which implies 
\begin{equation}\label{prima parte}
\rho_B(Z_i,\bar Z_i)=\begin{cases}-\sqrt{-1}\frac{3}{4}\left(1+\frac{\lvert C_{r}\rvert^2}{A_{p_r}B_{p_r}-\lvert C_{r}\rvert^2}\right)\quad & \mbox{if  there exists } r=1,\ldots, k \,\, \mbox{such that } i=p_r\,, \\
 -\sqrt{-1}\frac34 \quad & \mbox{ if } i\not\in \{p_1,\ldots, p_k\}\,.
\end{cases}
\end{equation}

We have 
$$
\begin{aligned}
\rho_3(Z_i,\bar Z_i)=&\sum_{j=1}^kg^{p_j\overline{s+p_j}}\omega([Z_i,\bar Z_i],[Z_{p_j},\bar W_{p_j}])
=-\frac{\sqrt{-1}}{2}\sum_{j=1}^kg^{p_j\overline{s+p_j}}\bar \lambda_{p_jp_j}\omega(Z_i+\bar Z_i, \bar W_{p_j}) \\ 
=& 
\begin{cases} 0\quad & \mbox{if } i\not\in\{p_1,\ldots, p_k \}\,, \\ \frac{1}{2} g^{i\overline{s+i}}\bar \lambda_{ii}C_{i} \quad  &\mbox{otherwise }. 
\end{cases}
\end{aligned}
$$
We compute the three addends in the expression of $\rho_4$ separately:
 \begin{eqnarray*}
 &&\begin{aligned}
 \omega([[Z_i,\bar Z_i]^{1,0}, W_{p_j}], \bar Z_{p_j})=&
 -\frac{1}{2}\lambda_{ip_j}\bar C_{p_j}
 =&\,\begin{cases} 0 \quad & \mbox{if  } i\not\in\{p_1,\ldots, p_k\}\quad \mbox{or }  i\ne p_j\,, \\ -\frac{1}{2}\lambda_{ii}	\bar C_i\quad & \mbox{otherwise }; 
 \end{cases}
 \end{aligned}\\
 &&\begin{aligned}
 \omega([Z_i,\bar Z_i],[W_{p_j},\bar Z_{p_j}])=&
 \frac{1}{2}\lambda_{p_jp_j}g_{\overline{i}s+p_j}
       =&\,\begin{cases}0 \quad& \mbox{if  } i\not\in\{p_1,\ldots, p_k\}\quad  \mbox{or }  \,\, i\ne p_j\,, \\ \frac{1}{2}\lambda_{ii}\bar C_i\quad & \mbox{otherwise }; \end{cases}
\end{aligned}\\
&&\begin{aligned}
\omega([[Z_i,\bar Z_i]^{0,1},\bar W_{p_j}], Z_{p_j})=&
\frac{1}{2}\bar \lambda_{ip_j}g_{\overline{s+p_j}p_j}
=&\begin{cases}0  \quad & \mbox{if } i\ne p_j\,, \\ \frac{1}{2}\bar \lambda_{ii}C_{i} \quad & \mbox{otherwise }. \end{cases}
\end{aligned}
\end{eqnarray*}
 It follows  
 $$
 \rho_3(Z_i,\bar Z_i)=\rho_4(Z_i,\bar Z_i)=0\quad  \mbox{ if } i\not\in\{p_1,\ldots,p_k\}\,,
 $$
 and, for $i\in\{p_1,\ldots, p_k\}$,
 $$
 \rho_3(Z_{i},\bar Z_{i})+\rho_4(Z_{i},\bar Z_{i})=-\frac{1}{2} g^{i\overline{s+i}}\bar \lambda_{ii} C_i
 -g^{s+i\overline{i}}\frac{1}{2}\lambda_{ii}\bar C_i +g^{s+i\overline{i}}\frac{1}{2}\lambda_{ii}\bar C_i+g^{\overline{s+i}i}\frac{1}{2}\bar \lambda_{ii}C_i=0\,. 
$$
Now, we focus on the calculation of $\rho_B(Z_i,\bar W_j)$. We have 
$$
\begin{aligned}
\rho_1(Z_i,\bar W_j)=&\, \sum_{a=1}^sg^{a\bar a}\bar \lambda_{ij}\left(-\frac{\sqrt{-1}}{2}\omega(\bar W_j, Z_a-\bar Z_a)+\omega([\bar W_j, \bar Z_a], Z_a)\right)\\
=&\, \begin{cases}0 \quad &\mbox{otherwise}\,,  \\ 
\sqrt{-1}g^{i\bar i}C_i\bar \lambda_{ii}\left(\frac{\sqrt{-1}}{2}-\bar \lambda_{ii}\right) \quad & \mbox{if } i=j\in\{p_1,\ldots, p_k\}\, ,
\end{cases}
\end{aligned}
$$
and since $\mathfrak{I}$ is abelian
$$
\rho_2(Z_i,\bar W_j)=0\, .
$$
Furthermore
$$
\begin{aligned}
\rho_3(Z_i,\bar W_j)=&\, 
  \sum_{j=1}^kg^{\overline {p_j} s+p_j  }\omega([[Z_i,\bar W_j]^{0,1},\bar Z_{p_j}], W_{p_j})=-\sqrt{-1}\sum_{j=1}^kg^{\overline {p_j} s+p_j  }\bar \lambda_{ij}\bar \lambda_{p_j p_j}g_{\overline{s+j}s+p_j}\\
=&\, \begin{cases}0 \quad & \mbox{otherwise}\,,  \\ 
 -\sqrt{-1}\bar \lambda_{jj}^2g^{\overline{j}s+j}B_j\quad & \mbox{if } i=j\in \{p_1,\ldots, p_k\}\end{cases}
\end{aligned}
$$
and 
$$
\begin{aligned}
\rho_4(Z_i,\bar W_j)=&\sum_{j=1}^kg^{s+p_j\bar p_j}\omega([Z_i,\bar W_j],[W_{p_j},\bar Z_{p_j}])=\sqrt{-1}\sum_{j=1}^kg^{s+p_j\bar p_j}\bar \lambda_{ij}\lambda_{p_jp_j}g_{\overline{s+j}s+p_j} \\ 
=&\, \begin{cases}0 \quad &\mbox{otherwise}\,,\\
 \sqrt{-1}g^{s+j\bar j}\bar \lambda_{jj}\lambda_{jj}B_j\quad & \mbox{if } i=j\in \{p_1,\ldots, p_k\}\, . \end{cases}
\end{aligned}
$$

It follows that $\rho_B(Z_i,\bar W_j)\ne 0 $ if and only if $i=j\in \{p_1,\ldots, p_k\}.$ In such a case, we have 
$$
\begin{aligned}
\rho_B(Z_j,\bar W_j)=&
-\sqrt{-1}\left(g^{s+j\overline{j}}B_j\left(\lvert\lambda_{jj}\rvert^2-\bar \lambda_{jj}^2\right)+g^{j\bar j}C_j\bar \lambda_{jj}\left(\frac{\sqrt{-1}}{2}-\bar \lambda_{jj}\right)\right)\,.
\end{aligned}
$$
Since 
$$
g^{s+j\bar j }B_j=-\frac{B_jC_j}{A_jB_j-\lvert C_j\rvert^2}\quad \mbox{and }\quad g^{j\bar j}C_j=\frac{B_jC_j}{A_jB_j-\lvert C_j\rvert^2}\,,
$$
we infer 
$$
\rho_B(Z_j,\bar W_j)=-\sqrt{-1}\left(\bar\lambda_{jj}\left(\frac{\sqrt{-1}}{2}-\bar\lambda_{jj}\right)-\left(\lvert\lambda_{jj}\rvert^2-\bar\lambda_{jj}^2\right)\right)\frac{B_jC_j}{A_jB_j-\lvert C_j\rvert^2}\,.
$$
Taking into account that $\lambda_{jj}=-\frac{\sqrt{-1}}{4}-\frac{c_{jj}}{2}$, we obtain
$$
\rho_B(Z_j,\bar W_j)=-\sqrt{-1}\left(-\frac{3}{16}-\frac{c_{jj}^2}{4}-\frac{\sqrt{-1}c_{jj}}{4}\right)\frac{B_jC_j}{A_jB_j-\lvert C_j\rvert^2}
$$
and the claim follows. 
\end{proof}

\begin{cor}\label{Cor 5.4}
Let $\omega$ be a left-invariant  pluriclosed  Hermitian metric on an Oeljeklaus-Toma manifold $M$. Then $\omega$ lifts to an algebraic expanding soliton of the pluriclosed flow on the universal covering of $M$ if and only if it  takes the following diagonal expression with respect to a coframe $\{\omega^1,\dots,\omega^s,\gamma^1,\dots,\gamma^s\}$ satisfying \eqref{eqsstr} and \eqref{lambdas}:
\begin{equation}\label{diagonal}
\omega=\sqrt{-1}\sum_{i=1}^sA\omega^i\wedge \bar \omega^i+ B_i\gamma^i\wedge \bar \gamma^i\,.
\end{equation}
\end{cor}

\begin{proof}  
%
Let $\omega$ be a pluriclosed left-invariant metric  on an Oeljeklaus-Toma manifold $M$. 
In view of  \cite[Section 7]{lauret}, $\omega$ lifts to an algebraic expanding soliton of the pluriclosed flow on the universal covering of $M$ if and only if 
$$
\rho^{1,1}_B(\cdot,\cdot)=c \omega(\cdot,\cdot)+\frac12\left(\omega(D\cdot,\cdot)+\omega(\cdot,D\cdot)\right)\,,
$$
for some  $c\in \R_{-}$ and some derivation $D$ of $\g$ such that $DJ=JD$.

\medskip 
Assume that $\omega$ takes the expression in formula \eqref{diagonal}. 
Proposition \ref{prop4.3} implies that $\rho_B$ is represented with respect to the basis $\{Z_1,\ldots, Z_s,W_1,\ldots, W_s\}$ by the matrix 
$$
P=-\frac{3}{4A} \begin{pmatrix}
{\rm I}_\mathfrak{h} & 0\\
0 & 0
\end{pmatrix}\,.
$$ 
Since 
$$
\frac{3}{4A}  \begin{pmatrix}
0 & 0\\
0 & {\rm I}_\mathfrak{I}
\end{pmatrix}
$$
induces a symmetric derivation on $\g$, $\omega$ lifts to an algebraic expanding soliton of the pluriclosed flow on the universal covering of $M$ and the first part of the claim follows. 

\medskip 
In order to prove the second part of the statement, we need some preliminary observations on derivations $D$ of $\mathfrak{g}$ that commute with $J$, i.e. such that 
$$
D(\mathfrak{g}^{1,0})\subseteq \mathfrak{g}^{1,0}\, ,\quad D(\mathfrak{g}^{0,1})\subseteq \mathfrak{g}^{0,1}\,.
$$
We can write
$$ 
DZ_i= \sum_{j=1}^s k^i_jZ_j+m^i_jW_j\quad \mbox{ and }\quad 
D\bar Z_i= \sum_{j=1}^sl^i_j\bar Z_j+ r^i_j\bar W_j\, .
$$
 Since $D$ is a derivation, we have, for all $i=1,\ldots, s$, 
 $$
 D[Z_i,\bar Z_i]=[DZ_i,\bar Z_i]+[Z_i,D\bar Z_i]\, .
 $$
 On the other hand 
 $$
 \begin{aligned}
 D[Z_i,\bar Z_i]=&\,
 -\frac{\sqrt{-1}}{2}\left(\sum_{j=1}^sk^i_jZ_j+l^i_j\bar Z_j+m^i_jW_j+r^i_j\bar W_j \right)\, ,\\
 [DZ_i,\bar Z_i]=&\,
 -\frac{\sqrt{-1}}{2}k^i_i(Z_i+\bar Z_i)-\sum_{j=1}^sm^i_j\lambda_{ij}W_j\,,\\
 [Z_i,D\bar Z_i]=&\,
 -\frac{\sqrt{-1}}{2}l^i_j(Z_i+\bar Z_i)+\sum_{j=1}^sr^i_j\bar \lambda_{ij}\bar W_j
 \end{aligned}
 $$
and 
 $$
 \begin{aligned}
 0=&\, D[Z_i,\bar Z_i]-[DZ_i,\bar Z_i]-[Z_i,D\bar Z_i]\\
 =&\, -\frac{\sqrt{-1}}{2}\sum_{j\ne i}k^i_jZ_j+l^i_j\bar Z_j+\frac{\sqrt{-1}}{2}l^i_iZ_i+\frac{\sqrt{-1}}{2}k^i_i\bar Z_i+\sum_{j=1}^sm^i_j\left(\lambda_{ij}-\frac{\sqrt{-1}}{2}\right)W_j-r^i_j\left(\frac{\sqrt{-1}}{2}+\bar \lambda_{ij}\right)\bar W_j
 \end{aligned}
 $$
 which forces $DZ_i, D\bar Z_i=0$, for all $i=1,\ldots, s$.  It follows that  $D_{|\mathfrak{h}}=0$. 
 
 \medskip 
Moreover, for all $I,I'\in \mathfrak{J}$, we have 
 $$
 0=D[I,I']=[DI,I']+[I,DI']\, ,
 $$
 which implies 
  $$
 [DI,I']=-[I,DI']\, .
 $$
Assume
$$
DW_i= \sum_{j=1}^s k^{s+i}_jZ_j+m^{s+i}_jW_j \quad \mbox{ and }\quad 
D\bar W_i= \sum_{j=1}^sl^{s+i}_j\bar Z_j+ r^{s+i}_j\bar W_j\,,
$$
then
 $$
 [DW_i,\bar W_i]=\sum_{j=1}^sk^{s+i}_j[Z_j,\bar W_i]\in\mathfrak{J}^{0,1}
 \quad \mbox{ and }\quad 
 [W_i,D\bar W_i]=\sum_{j=1}^sl^{s+i}_j[W_i,\bar Z_j]\in \mathfrak{J}^{1,0}\,.
 $$
 This implies
$$
DW_i= \sum_{j=1}^sm^{s+i}_jW_j\, ,\quad 
D\bar W_i= \sum_{j=1}^s r^{s+i}_j\bar W_j\,,
$$
i.e. $D(\mathfrak{J})\subseteq\mathfrak{J}$.
Moreover, for all $i=1,\ldots, s$, we have that 
 $$
 D[Z_i, W_i]=-\lambda_{ii}DW_i=-\sum_{j=1}^s\lambda_{ii}m^{s+i}_jW_j\,,
 $$
 while $[DZ_i,W_i]=0$ and 
 $$
 [Z_i, DW_i]=-\sum_{j=1}^sm^{s+i}_j\lambda_{ij}W_j\,.
 $$
 Using again the fact that $D$ is a derivation, we have 
$$
 DW_i=\sum_{j\in J_i}m_jW_j
 $$
 where 
 $$
 J_i=\{j\in \{1,\ldots, s\}\quad | \quad \lambda_{ii}=\lambda_{ij}\}\, .
 $$
 With analogous computations, we infer 
 $$
 D\bar W_i=\sum_{j\in J_i}r^{s+i}_j\bar W_j\,.
 $$
 Clearly, $i\in J_i$. On the other hand, for all $i=1,\ldots, s$, we know that $\Im\mathfrak{m}(\lambda_{ii})\ne 0 $, while, for all $i\ne j$, $\lambda_{ij}\in\R.$ This guarantees that, for all $i=1,\ldots, s$, 
 $$
 J_{i}=\{i\}\,.
 $$
 This allows us to write 
$$
DW_i= m^{s+i}_iW_i\,,\quad 
D\bar W_i= r^{s+i}_i\bar W_i\,.
$$
From the relations above, we obtain that 
$$
{\rm Der}(\mathfrak{g})^{1,0}=\{E\in {\rm End}(\mathfrak{g})^{1,0}\quad|\quad \mathfrak{h}\subseteq\ker(E)\,, \,\, E(\langle W_i\rangle)\subseteq \langle W_i\rangle\,, \quad \mbox{for all } i =1,\ldots, s\}\,.
$$
First of all,   we suppose that $\omega$ is a pluriclosed Hermitian metric  which   takes the following diagonal expression with respect to a coframe $\{\omega^1,\dots,\omega^s,\gamma^1,\dots,\gamma^s\}$ satisfying \eqref{eqsstr} and \eqref{lambdas}:
$$
\omega=\sqrt{-1}\sum_{i=1}^sA_i\omega^i\wedge \bar \omega^i+ B_i\gamma^i\wedge \bar \gamma^i\,. 
$$  such that there exist $i, j\in\{1,\ldots, s\}$ such that $A_i\ne A_j$ and we suppose that $\omega$ is an algebraic soliton. Thanks to the facts regarding derivations proved before, we have that 
$$
\begin{aligned}
-\sqrt{-1}\frac34=\rho_B(Z_i,\bar Z_i)=&\,  c\omega(Z_i,\bar Z_i)+ \frac12\left(\omega(DZ_i,\bar Z_i)+\omega(Z_i, D\bar Z_i)\right)=\sqrt{-1}cA_i\,,\\
-\sqrt{-1}\frac34=\rho_B(Z_j,\bar Z_j)=&\,  c\omega(Z_j,\bar Z_j)+ \frac12\left(\omega(DZ_j,\bar Z_j)+\omega(Z_j, D\bar Z_j)\right)=\sqrt{-1}cA_j\,,
\end{aligned}
$$ which is impossible, since $A_i\ne A_j$.

Now suppose that $\omega$ is a pluriclosed metric on $M$ which is not diagonal. So, we suppose that  there exists  $\tilde{j}=1,\ldots, s$ such that $C_{\tilde{j}}\ne 0.$   Then, assume that there exist a constant $c\in\mathbb{R}$ and $D\in {\rm Der(}\mathfrak{g})$   such that 
 $$
(\rho_B)^{1,1}(\cdot, \cdot)=c\omega(\cdot, \cdot)+ \frac12\left(\omega(D\cdot,\cdot)+\omega(\cdot,D\cdot)\right)\,,\quad DJ=JD\,.
$$

On the other hand
$$
\begin{aligned} 
0=\rho_B(W_{\tilde{j}},\bar W_{\tilde{j}})=&\,  c\omega(W_{\tilde{j}},\bar W_{\tilde{j}})+ \frac12\left(\omega(DW_{\tilde{j}},\bar W_{\tilde{j}})+\omega(W_{\tilde{j}}, D\bar W_{\tilde{j}})\right)=\sqrt{-1}cB_{\tilde{j}}+\frac{\sqrt{-1}}{2}(r_{\tilde{j}}^{s+\tilde{j}}+m_{\tilde{j}}^{s+\tilde{j}})B_{\tilde{j}}\,,  \\
  \rho_B(Z_{\tilde{j}},\bar W_{\tilde{j}})= & c\omega(Z_{\tilde{j}},\bar W_{\tilde{j}})+ \frac{1}{2}\left(\omega(DZ_{\tilde{j}},\bar W_{\tilde{j}})+\omega(Z_{\tilde{j}}, D\bar W_{\tilde{j}})\right)=\sqrt{-1}cC_{\tilde{j}}+\frac{\sqrt{-1}}{2}r^{s+\tilde{j}}_{\tilde{j}}C_{\tilde{j}}\, ,\\
 \rho_B(\bar Z_{\tilde{j}}, W_{\tilde{j}})=&  c\omega(\bar Z_{\tilde{j}}, W_{\tilde{j}})+ \frac12\left(\omega(D\bar Z_{\tilde{j}}, W_{\tilde{j}})+\omega(\bar Z_{\tilde{j}}, D W_{\tilde{j}})\right)=-\sqrt{-1}c\bar C_{\tilde{j}}-\frac{\sqrt{-1}}{2}m^{s+\tilde{j}}_{\tilde{j}}\bar C_{\tilde{j}}\,,
 \end{aligned}
 $$
   which implies that 
   $$
   c=-\frac12(r_{\tilde{j}}^{s+{\tilde{j}}}+m_{\tilde{j}}^{s+\tilde{j}})\, ,
   $$
   On the other hand, 
  $$
   \rho_B(Z_{\tilde{j}},\bar W_{\tilde{j}})=\sqrt{-1}KC_{\tilde{j}}\, ,
  $$
  where $$
  K=\left(\frac{3}{16}+\frac{c_{\tilde{j}\tilde{j}}^2}{4}+\frac{\sqrt{-1}c_{\tilde{j}\tilde{j}}}{4}\right)\frac{B_{\tilde{j}}}{A_{\tilde{j}}B_{\tilde{j}}-\lvert C_{\tilde{j}}\rvert^2}\,. 
  $$
  Then, 
  $$
  K=c+\frac12r_{\tilde{j}}^{s+\tilde{j}}=-\frac12m_{\tilde{j}}^{s+\tilde{j}} 
  $$
  and 
  $$
  \bar K=c+\frac12m_{\tilde{j}}^{s+\tilde{j}}=-\frac12r_{\tilde{j}}^{s+\tilde{j}}\, .
  $$
  From this we obtain that 
  $$
  c=K+\bar K=2\Re\mathfrak{e}(K)>0\,.
  $$
   On the other hand,  we have 
   $$
   -\sqrt{-1}\frac34\left(1+\frac{\lvert C_{\tilde{j}}\rvert^2}{A_{\tilde{j}}B_{\tilde{j}}-\lvert C_{\tilde{j}}\rvert^2}\right)= \rho_B(Z_{\tilde{j}},\bar Z_{\tilde{j}})= c\omega(Z_{\tilde{j}},\bar Z_{\tilde{j}})+ \frac12\left(\omega(DZ_{\tilde{j}},\bar Z_{\tilde{j}})+\omega(Z_{\tilde{j}},D\bar Z_{\tilde{j}})\right)=\sqrt{-1}cA_{\tilde{j}}\,,
   $$
   which implies that $c$ must be negative. From this the claim follows.
\end{proof}

\begin{cor}\label{Cor4.5}
Let $\omega$ be a pluriclosed  Hermitian metric on an Oeljeklaus-Toma manifold which takes the form $(\ref{gPC})$. Then the pluriclosed flow starting from $\omega$ is equivalent to the following system of ODEs:
\begin{equation}
\begin{cases} A_i'=\frac{3}{4}\quad &\mbox{if } i\not\in\{p_1,\ldots, p_k\}\,, \\
 A_{p_r}'=\frac{3}{4}\left(1+\frac{\lvert C_{r}\rvert^2}{A_{p_r}B_{p_r}-\lvert C_r\rvert^2}\right)\quad& \mbox{for all } r=1,\ldots, k\, , \\ 
B_j'=0\quad &\mbox{for all }j=1,\ldots, s\, , \\ 
C_r'= -\left(\frac{3}{16}+\frac{c_{p_rp_r}^2}{4}+\frac{\sqrt{-1}c_{p_rp_r}}{4}\right)\frac{B_{p_r}C_r}{A_{p_r}B_{p_r}-|C_r|^2}\quad &\mbox{for all } r=1,\ldots, k\, .\end{cases}
\end{equation} 
Moreover, $\lvert C_r\rvert$ is bounded, for all $r =1,\ldots, k$, the solution exists for all  $ t\in [0,+\infty)$ and $A_i\sim\frac{3}{4}t$, as $t\to+\infty,$ for all  $ i=1,\ldots, s.$ 

In particular, 
$$
\frac{\omega_t}{1+t}\to 3\omega_{\infty}
$$ 
as $t\to \infty$. 
\end{cor}
\begin{proof}
Observe that,  for every $ r\in \{1,\ldots, k\}$,
$$
\begin{aligned}
(\lvert C_r\rvert^2)'=&\, 
- \left(\frac{3}{8}+\frac{c_{jj}^2}{2}\right)\frac{B_{p_r}\lvert C_r\rvert^2}{A_{p_r}B_{p_r}-\lvert C_r\rvert^2}\le 0\, ,
\end{aligned}
$$
which guarantees  that $\lvert  C_r\rvert^2$ is bounded. On the other hand, denote, for all $ r=1,\ldots, k$,
 $$
 u_r=A_{p_r}B_{p_r}- \lvert C_r\rvert^2.
 $$
  We have that
  $$
  u_r'=A'_{p_r}B_{p_r}-(\lvert C_r\rvert^2)'=\frac{3}{4}B_{p_r}+\left(\frac{9}{8}+\frac{c^2_{p_rp_r}}{2}\right)\frac{B_{p_r}\lvert C_r\rvert^2}{A_{p_r}B_{p_r}-\lvert C_r\rvert^2}\ge 0\, .
 $$
 This guarantees $$
 A_{p_r}'=\frac{3}{4}\left(1+\frac{\lvert C_{r}\rvert^2}{A_{p_r}B_{p_r}-\lvert C_r\rvert^2}\right)\le \frac{3}{4}\left(1+\frac{K}{u_r(0)}\right)\,,
 $$
  where $K>0$ such that $\lvert C_r\rvert^2\le K$,  for all $ t\ge 0$. This implies the long-time existence.  As regards the last part of the statement, it is sufficient to prove that
   $$
  \lim_{t\to+\infty} \frac{\lvert C_r\rvert^2}{u_r}=0\,.
  $$
But, 
$$
u_r'\ge \frac{3}{4}B_{p_r}\,.
$$ 
So,
$$
u_r\ge \frac{3}{4}B_{p_r}t+u_r(0)\to +\infty\, ,\,\,\, t\to+\infty\, .
$$
Then, $$
\lim_{t\to+\infty}u_r(t)=+\infty\, ,
$$ 
and, since $\lvert C_r\rvert^2$ is bounded, the assertion follows. \end{proof}

\begin{proof}[Proof of Theorem $\ref{main2}$]
Let $\omega$ be a left-invariant pluriclosed metric on an Oeljeklaus-Toma manifold. 
Corollary \ref{Cor4.5} implies that pluriclosed flow starting from $\omega$ has a long-time solution $\omega_t$ such that
$$
\frac{\omega_t}{1+t}\to3\omega_\infty \quad \mbox{ as }\quad  t\to \infty\,. 
$$
We show that $\frac{\omega_t}{1+t}$ satisfies conditions 1,2,3 in Proposition \ref{GH}. Here we 
denote by $|\cdot|_t$ the norm induced by $\omega_t$.  

\smallskip
 Taking into account that 
$$
\omega_{t|\mathfrak I\oplus \mathfrak I}=\omega_{0|\mathfrak I\oplus \mathfrak I}\, ,
$$
 condition 2 follows.
\smallskip

Thanks to the fact that condition 2 holds, 
$$
\omega_{t|\mathfrak h\oplus \mathfrak h}=
\sum_{i=1}^sA_i(t)\omega^i\wedge \bar \omega^i
$$
with $\frac{A_i(t)}{1+t}\to \frac34$ as $t\to \infty$ and there exist $C,T>0$ such that, for every vector $v\in \mathfrak h$, 
$$
\frac{1}{\sqrt{1+t}}|v|_t\leq C|v|_{0}\,, 
$$
for every $t\geq T$, condition 1 is satisfied.

\smallskip 
In order to prove Condition 3, let $\epsilon, \ell >0$ and let $\gamma$ be a curve in $M$ tangent to $\mathcal H$ which is parametrized by arclength with respect to $3\omega_\infty$ and such that $L_{\infty}(\gamma)<\ell$. 
Let $v=\dot \gamma$ and $T>0$ such that 
$$
\left\vert \frac{A_i(t)}{1+t}-\frac{3}{4}\right\vert\leq \frac{3\epsilon^2}{4\ell^2}\,,
$$
for $t\ge  T$. 
Then
$$
\left\vert\frac{1}{1+t}|v|^2_t-|v|^2_{\infty}\right\vert\le \sum_{i=1}^{s}\left\vert\frac{A_i(t)}{1+t}-\frac34 \right\vert\lvert v_i\rvert^2\leq  \frac{\epsilon^2}{\ell^2}
$$
and 
$$
|L_{t}(\gamma)-L_{\infty}(\gamma)|\le \int_{0}^{b}\left\lvert\frac{1}{\sqrt{1+t}}|\dot \gamma |_t-|\dot \gamma|_\infty\right\rvert da\leq \frac{\epsilon}{\ell}b\leq \epsilon \,, 
$$
since $b\leq \ell$.

\medskip 
Now we show the last part of the statement, using the same argument as in Proposition \ref{main1},  and we prove that  $(\mathbb H^s\times\mathbb C^s, \frac{\omega_t}{1+t})$ converges in the Cheeger-Gromov sense to $(\mathbb H^s\times\mathbb C^s, \tilde{\omega}_{\infty})$  where $\tilde{\omega}_{\infty}$  is an  algebraic soliton. Again, here we are identifying $\omega_t$ with its pull-back onto $\mathbb H^s\times\mathbb C^s$ and  we are fixing as base point the identity element of $\mathbb H^s\times\mathbb C^s$. It is enough to construct  a  1-parameter family of biholomorphisms $\{\varphi_t\}$  of $\mathbb H^s\times\mathbb C^s$ such that 
 $$
 \varphi_t^*\frac{\omega_t}{1+t}\to \tilde{\omega}_{\infty}\,.
 $$
 As we already observed, since $\mathfrak{I}$ is abelian the endomorphism represented by the matrix
 $$
 D=\begin{pmatrix} 0 & 0 \\ 0 & I_{\mathfrak{I}}\end{pmatrix}
 $$
 is a derivation of $\mathfrak{g}$ that commutes with the complex structure $J$. 
 Then, we can consider 
 $$
 d\varphi_t= \exp(s(t)D)=\begin{pmatrix}I_{\mathfrak{h}}& 0 \\ 0 & e^{s(t)}I_{\mathfrak{I}}\end{pmatrix}\in {\rm Aut}(\mathfrak{g}, J)
 $$
 where $s(t)=\log(\sqrt{1+t})$.  Using $d\varphi_t$,  we can define 
 $$
 \varphi_t\in {\rm Aut}(\mathbb{H}^s\times\mathbb{C}^s, J)\, . 
 $$
For $i=1,\ldots, s$ we have 
 $$
  \begin{aligned}
  \frac{1}{1+t}(\varphi_t^*\omega_t)(Z_i,\bar Z_i)=&\, \frac{1}{1+t}\omega_t(Z_i,\bar Z_i)\to \frac34\sqrt{-1}\, , \quad \mbox{as } t\to \infty\,, \\  \frac{1}{1+t}(\varphi_t^*\omega_t)(Z_i,\bar W_i)=&\, \frac{1}{\sqrt{1+t}}\omega_t(Z_i, \bar W_i )\to 0\, ,  \quad \mbox{as } t\to \infty\,, \\
  \frac{1}{1+t}(\varphi_t^*\omega_t)(W_i,\bar W_i)=&\,\omega_t(W_i,\bar W_i)=\sqrt{-1}B_i(0)\,. 
  \end{aligned}
   $$
   Then, 
   $$
   \frac{1}{1+t}\varphi_t^*\omega_t\to \tilde{\omega}_{\infty}\, , \quad \mbox{as } t\to \infty\,,
   $$
   where 
$$
\tilde{\omega}_{\infty}=3\,\omega_{\infty}+\omega_{|\mathfrak{I}\oplus \mathfrak{I}}\,.
$$
Notice that $\tilde{\omega}_{\infty}$ is an algebraic soliton diagonal since $\omega_{|\mathfrak{I}\oplus \mathfrak{I}}$ is diagonal in view of Proposition \ref{ch}.
\end{proof}

\section{A generalization to semidirect product of Lie algebras} 
From the viewpoint of Lie groups, the algebraic structure of  Oeljeklaus-Toma manifolds is quite rigid and some of the 
results in the previous sections can be generalized to semidirect product of Lie algebras. 

\medskip 
In this section we consider a Lie algebra $\g$  which is a semidirect product  of Lie algebras
$$
\mathfrak g=\mathfrak h \ltimes_{\lambda} \mathfrak I\,,
$$
where $\lambda\colon \mathfrak h\to {\rm Der}(\mathfrak I)$ is a representation. We further assume that $\g$ has a complex structure of the form 
$$
J=J_{\mathfrak h}\oplus J_{\mathfrak I}
$$
where  $J_{\mathfrak h}$ and  $J_{\mathfrak I}$ are complex structures on $\mathfrak h$ and $\mathfrak I$, respectively. 

\medskip 
The following assumptions are all satisfied in the case of an Oeljeklaus-Toma manifold: 
\begin{enumerate}
\item[i.] $\mathfrak h$ has $(1,0)$-frame such that $\{Z_1,\dots,Z_r\}$ such that $[Z_k,\bar Z_k]=\,-\frac{\sqrt{-1}}{2}(Z_k+\bar Z_k)$,  for all $k=1,\dots ,r$ and the other brackets vanish;

\vspace{0.2cm} 
\item[ii.] $\mathfrak I$ is a $2s$-dimensional abelian Lie algebra and $J_{\mathfrak I}$ is a complex structure on $\mathfrak I$; 

\vspace{0.2cm} 
\item[iii.] $\lambda(\mathfrak h^{1,0})\subseteq {\rm End}(\mathfrak I)^{1,0}$;

\vspace{0.2cm} 
\item[iv.] $\mathfrak I$ has a $(1,0)$-frame $\{W_1,\dots W_s\}$ such that $\lambda(Z)\cdot\bar W_r=\lambda_r(Z)\bar W_r$,  for every $r=1,\dots,s$, where $\lambda_r\in \Lambda^{1,0}(\mathfrak{h})$;

\vspace{0.2cm} 
\item[v.] $\sum_{a=1}^s\Im\mathfrak m(\lambda_a(Z_i))$ is constant on $i$. 

\vspace{0.2cm} 
\item[vi.]  $\mathfrak I$ has a $(1,0)$-frame $\{W_1,\dots W_s\}$ such that $\lambda(Z)\cdot W_r=\lambda'_r(Z) W_r$,  for every $r=1,\dots,s$, where $\lambda_r'\in \Lambda^{1,0}(\mathfrak{h})$ and $\sum_{a=1}^s\Im\mathfrak m(\lambda_a'(Z_i))$ is constant on $i$. 
\end{enumerate} 

Note that condition i. is equivalent to require that 
$\mathfrak h=\underbrace{\mathfrak{f}\oplus \dots \oplus  \mathfrak{f}}_{\mbox{$r$-times}}$ equipped with the complex structure $J_\mathfrak{h}=\underbrace{J_{\mathfrak f}\oplus \dots \oplus J_{\mathfrak f}}_{\mbox{$r$-times}}$, while in condition iv. the existence of $\{W_r\}$ and $\lambda_r$ is equivalent to require that 
$$
\lambda(Z)\circ \lambda(Z')=\lambda(Z') \circ \lambda(Z)\,, 
$$ 
for every $Z,Z'\in \mathfrak h^{1,0}$. 

\medskip 

The computations in Section \ref{Sec4} can be used to study solutions to the flow
\begin{equation}\label{fpcf}
\partial_t\omega_t=-\rho_B^{1,1}(\omega_t)
\end{equation}
in semidirect products of Lie algebras (this flow coincides to the pluriclosed flow only when the initial metric is pluriclosed). We have the following 

\begin{prop}\label{ginocchio}
Let $\mathfrak g=\mathfrak h \ltimes_{\lambda} \mathfrak I$ be a semidirect product of Lie algebras equipped with a
splitting complex structure $J=J_{\mathfrak h}\oplus J_{\mathfrak I}$ and let $\omega$ be a Hermitian metric on $\g$ making $\mathfrak h$ and $\mathfrak I$ orthogonal. Then the Bismut Ricci-form of $\omega$ satisfies $\rho^{1,1}
_{B|\mathfrak h\oplus \mathfrak I}=\rho^{1,1}
_{B|\mathfrak I\oplus \mathfrak I}=0.$

\smallskip 
If i-iv hold and $\omega_{|\mathfrak h\oplus \mathfrak h}$ is diagonal with respect to the frame  $\{Z_i\}$ then the $(1,1)$-component of the Bismut-Ricci form of $\omega$ does not depend on $\omega$ and the solution to the flow \eqref{fpcf} starting from $\omega$ takes the following expression
$$
\omega_t=\omega-t\rho^{1,1}_B(\omega)\,.
$$

\smallskip 
If ${\rm i-iv}$  and ${\rm vi}$ hold and $\omega_{|\mathfrak h\oplus \mathfrak h}$ is a multiple of the canonical metric with respect to the frame  $\{Z_i\}$, then $\omega$ is a soliton for  flow \eqref{fpcf} with cosmological constant $c=\frac{1}{2}+\sum_{a=1}^s\Im\mathfrak m(\lambda_a'(Z_i))$.  
\end{prop}

The previous Proposition does not cover the case when properties i-iv are satisfied and the restriction to $\mathfrak h\oplus \mathfrak h$ of the initial Hermitian inner product 
$$
\omega=\sqrt{-1}\sum _{a,b=1}^r g_{a\bar b}\omega^a\wedge 	\bar \omega^b+\sqrt{-1}\sum _{a,b=1}^s
g_{r+a\overline{r+ b}}\gamma^a\wedge 	\bar \gamma^b
$$ 
is not  diagonal with respect to $\{Z_i\}$. In this case flow \eqref{fpcf} evolves only the components
$g_{i\bar i}$ of $\omega$ along $\omega^i\wedge \bar \omega^i$ via the ODE 
 
$$ 
\partial_{t}g_{i\bar i}=\frac14 \sum_{a=1}^r g^{\bar a a}\Re	\mathfrak{e}\,g_{i\bar a}  
-\frac12 \sum_{c,d=1}^s g^{\overline{r+d} r+ c} \left\lbrace   \omega([Z_i,W_c],\bar W_d) +
\omega([\bar Z_i,\bar W_c],W_d)  \right\rbrace 
$$
where $g_{i\bar i}$ depends on $t$. Note that the quantities $-\frac12 \sum_{c,d=1}^s g^{\overline{r+d} r+ c} \left\lbrace   \omega([Z_i,W_c],\bar W_d) +
\omega([\bar Z_i,\bar W_c],W_d) \right\rbrace$ appearing in the evolution of $g_{i\bar i}$  are independent on $t$.   
 
 The same computations as in Section \ref{Sec3} imply the following 
\begin{prop}
Let $\mathfrak g=\mathfrak h \ltimes_{\lambda} \mathfrak I$ be a semidirect product of Lie algebras equipped with a
splitting complex structure $J=J_{\mathfrak h}\oplus J_{\mathfrak I}$. Assume that properties ${\rm i, ii, iii}$ are satisfied and let $\omega$ be a left-invariant Hermitian metric on $\g$. Then 
$$
\rho_{C | \mathfrak I 	\oplus \mathfrak I }=\rho_{C|\mathfrak h \oplus \mathfrak I}=0\,, 
$$
while $ \rho_{C |\mathfrak h\oplus\mathfrak h}$ is diagonal with respect to $\{Z_1,\ldots, Z_r\}$. 

\smallskip 
If further also ${\rm iv}.$ holds, then
$$
\rho_C(Z_i,\bar Z_i)=-\sqrt{-1}\left(\frac12-\sum_{a=1}^s\Im\mathfrak m(\lambda_{a}(Z_i))\right), \quad \mbox{for all }  i=1,\ldots, r\,.
$$

\smallskip
If, in addition, v. holds, then $\omega$ is a soliton for the Chern-Ricci flow with cosmological constant $c=\frac{1}{2}-\sum_{a=1}^s\Im\mathfrak m(\lambda_a(Z_i))$ if and only if $\omega_{\mathfrak{h}\oplus \mathfrak h}$ is a multiple of the canonical metric on $\mathfrak h $ with respect to the frame $\{Z_i\}$ and $\omega_{\mathfrak h\oplus\mathfrak J }=0$. 
\end{prop}


\begin{thebibliography}{12}

\bibitem{A}
D. Angella, A. Dubickas, A. Otiman, J. Stelzig: On metric and cohomological properties of Oeljeklaus-Toma manifolds. {\tt arXiv:2201.06377}. 


\bibitem{AT}
 D. Angella, V. Tosatti, Leafwise flat forms on Inoue-Bombieri surfaces. {\tt arXiv:2106.16141.}

\bibitem{AL}
R.M. Arroyo, R.A. Lafuente,  The long-time behavior of the homogeneous pluriclosed flow.  {\em Proc. Lond. Math. Soc.} (3), {\bf 119}  (2019)(1): 266--289.

\bibitem{B}   J.-M. Bismut, A local index theorem for non-K\"ahler manifolds.  {\em Math. Ann.} {\bf 284} (1989), no. 4, 681--699.


\bibitem{boling}
J. Boling, Homogeneous Solutions of Pluriclosed Flow on Closed Complex Surfaces. {\em J. Geom. Anal.} {\bf 26} (2016), no. 3, 2130--2154.



\bibitem{EFV2} 
N. Enrietti, A. Fino, L. Vezzoni,
The pluriclosed flow on nilmanifolds and Tamed symplectic forms. {\em J. Geom. Anal.} {\bf 25} (2015), no. 2, 883--909.   

\bibitem{FTWZ}
S. Fang, V. Tosatti, B. Weinkove, T. Zheng, Inoue surfaces and the Chern-Ricci
flow. {\em J. Funct. Anal.} {\bf 271} (2016), no. 11, 3162--3185.


\bibitem{FKV}
A. Fino, H. Kasuya, L. Vezzoni, SKT and tamed symplectic structures on solvmanifolds. {\em Tohoku Math. J. (2)} {\bf 67} (2015), no. 1, 19--37.

%
%
\bibitem{mario}
M. Garcia-Fernandez, J. Jordan, J. Streets, Non-K\"ahler Calabi-Yau geometry and pluriclosed flow. {\tt arXiv:2106.13716}. 
 


\bibitem{Gill}
M. Gill, Convergence of the parabolic complex Monge-Amp\`ere equation on compact Hermitian manifolds. {\em Comm. Anal. Geom.} {\bf19} (2011), 277--303

\bibitem{In}
M. Inoue, On surfaces of Class $VII_0$. {\em Invent. Math.} {\bf 24} (1974), no.4, 269--320.

\bibitem{JS}
J. Jordan, J. Streets, On a Calabi-type estimate for pluriclosed flow. {\em Adv. Math.},
{\bf 366} (2020), Article ID: 107097, p.18.

\bibitem{Kas}
H. Kasuya, Vaisman metrics on solvmanifolds and Oeljeklaus-Toma manifolds. {\em Bull. Lond. Math. Soc.} {\bf 45} (2013), no. 1, 15--26.

\bibitem{lauret}
J. Lauret, Curvature flows for almost-hermitian Lie groups. {\em  Trans. Amer. Math. Soc.} {\bf 367} (2015), no. 10, 7453--7480.

\bibitem{LauretCG}
J. Lauret, Convergence of homogeneous  manifolds, {\em J. Lond. Math. Soc.}, II. Ser. {\bf 86} (2012), No. 3, 701--727.

\bibitem{lauret2}
J. Lauret, E.A. Rodr\'iguez Valencia, On the Chern-Ricci flow and its solitons for Lie
group. {\em Math. Nachr.} {\bf 288} (2015), no. 13, 1512--1526.

\bibitem{OT}
K. Oeljeklaus, M. Toma,  Non-K\"ahler compact complex manifolds associated to number fields. {\em Ann. Inst. Fourier (Grenoble).} {\bf 55} (2005), no. 1, 161--171

\bibitem{Otiman}
A. Otiman, Special Hermitian metrics on Oeljeklaus-Toma manifolds. {\tt arXiv:2009.02599. }

\bibitem{PV}
M. Pujia, L. Vezzoni, A remark on the Bismut-Ricci form on $2$-step nilmanifolds. {\em C. R. Math. Acad. Sci. Paris} {\bf 356} (2018), no. 2, 
222--226.

\bibitem{streets-tian}
J. Streets,  G. Tian, Hermitian curvature flow. {\em J. Eur. Math. Soc. (JEMS)} {\bf 13} (2011), no. 3, 601--634.

\bibitem{streets-tian2}
J. Streets,  G. Tian, A parabolic flow of pluriclosed metrics. {\em Int. Math. Res. Notices} (2010), 3101--3133.

\bibitem{streets-tian4}
J. Streets, G. Tian, Regularity results for pluriclosed flow. {\em Geom. Topol.} {\bf 17} (2013), no. 4, 2389--2429\,.

\bibitem{Streets3}
J. Streets, Classification of solitons for pluriclosed flow on complex surfaces. {\em Math. Ann. }, {\bf 375} (2019), no. 3--4, 1555--1595.


\bibitem{Streets1}  
J. Streets, Pluriclosed flow, Born-Infeld geometry, and rigidity results for generalized K\"ahler manifolds. {\em Comm. Partial Differential Equations} {\bf 41} (2016), no. 2, 318--374. 

\bibitem{Streets4}
J. Streets, Pluriclosed flow and the geometrization of complex surfaces. {\em Prog. Math.}, {\bf 333} (2020), 471--510.


\bibitem{Streets2}
J. Streets, Pluriclosed flow on generalized K\"ahler manifolds with split tangent bundle. {\em J. Reine Angew. Math.}, {\bf 739}(2018), 241--276.

\bibitem{Streets}  
J. Streets, Pluriclosed flow on manifolds with globally generated bundles. {\em Complex Manifolds} {\bf 3} (2016), 222--230. 

\bibitem{TWJDG}
V. Tosatti, B. Weinkove, On the evolution of a Hermitian metric by its Chern-Ricci
form. {\em J. Differential Geom.} {\bf 99} (2015), no.1, 125--163.

\bibitem{TWCom}  V. Tosatti, B. Weinkove, The Chern-Ricci flow on complex surfaces.  
{\em Compos. Math.} {\bf 149} (2013), no. 12, 2101--2138.


\bibitem{sima} S. Verbitsky, Surfaces on Oeljeklaus-Toma Manifolds. {\tt arXiv:1306.2456}. 

\bibitem{luigiproc}  L. Vezzoni, A note on canonical Ricci forms on 2-step nilmanifolds. {\em Proc. Amer. Math. Soc.} {\bf 141} (2013), no. 1, 325--333.

\bibitem{Zheng}
T. Zheng, The Chern-Ricci flow on Oeljeklaus-Toma manifolds, {\em Canad. J. Math.} {\bf 69}
(2017), no. 1, 220--240.

\end{thebibliography}
\end{document}